\documentclass[a4paper,leqno,12pt]{amsart}
\usepackage[all]{xy}
\usepackage{xspace}
\usepackage{amsmath}
\usepackage{amstext}
\usepackage{amsfonts}
\usepackage[mathscr]{euscript}
\usepackage{amscd}
\usepackage{latexsym}
\usepackage{amssymb}
\newenvironment{myeq}[1][]
{\stepcounter{thm}\begin{equation}\tag{\thethm}{#1}}
{\end{equation}}

\newcommand{\mydiagram}[2][]
{\stepcounter{thm}\begin{equation}
     \tag{\thethm}{#1}\vcenter{\xymatrix{#2}}\end{equation}}
\newcommand{\sect}{\setcounter{thm}{0}\section}
\newcommand{\secta}{\setcounter{thm}{0}\section*}
\numberwithin{equation}{section}

\theoremstyle{plain}
\swapnumbers
    \newtheorem{thm}{Theorem}[section]
    \newtheorem{rthm}[thm]{Resolution Theorem}
    \newtheorem{drthm}[thm]{Dual Resolution Theorem}
    \newtheorem{prop}[thm]{Proposition}
    \newtheorem{lemma}[thm]{Lemma}
    \newtheorem{hlemma}[thm]{Hauptlemma}
    \newtheorem{tlemma}[thm]{Transport Lemma}

    \newtheorem{subsec}[thm]{}
    \newtheorem*{thma}{Theorem A}
    \newtheorem*{thmb}{Theorem B}
\theoremstyle{definition}
    
    \newtheorem{defn}[thm]{Definition}
    \newtheorem{example}[thm]{Example}

    \newtheorem{notn}[thm]{Notation}
    \newtheorem{conv}[thm]{Convention}
\theoremstyle{remark}
        \newtheorem{remark}[thm]{Remark}
    \newtheorem{ack}[thm]{Acknowledgements}
\newcommand{\lra}{\longrightarrow}
\newcommand{\lla}{\longleftarrow}
\newcommand{\xra}[1]{\xrightarrow{#1}}

\newcommand{\hra}{\hookrightarrow}
\newcommand{\epic}{\to\hspace{-5 mm}\to}

\newcommand{\hsp}{\hspace{10 mm}}
\newcommand{\hs}{\hspace*{5 mm}}
\newcommand{\hsm}{\hspace*{3 mm}}
\newcommand{\vsm}{\vspace*{2 mm}}
\newcommand{\vsn}{\vspace{1 mm}}

\newcommand{\sotimes}{\overline{\otimes}}
\newcommand{\sbu}{\sb{\bullet}}
\newcommand{\ubu}{\sp{\bullet}}
\newcommand{\wA}{\widehat{A}}
\newcommand{\vare}{\varepsilon}
\newcommand{\omi}{0}
\newcommand{\ove}{\overline{e}}
\newcommand{\ovH}{\overline{H}}
\newcommand{\ovp}{\overline{p}}
\newcommand{\dA}{\mathcal{A}}

\newcommand\dB{\mathcal{B}}

\newcommand{\dC}{{\mathcal C}}
\newcommand{\K}{{\mathcal K}}
\newcommand{\dL}{\mathcal{U}}
\newcommand{\M}{{\mathcal M}}
\newcommand{\OO}{{\mathcal O}}
\newcommand{\SB}{\mathcal{S}\sb{\Box}}
\newcommand{\TT}{\mathbf{Top}}
\newcommand{\Ta}{\TT\sb{\ast}}
\newcommand{\Ws}{W\sb{\ast}}
\newcommand\dX{\mathcal{X}}
\newcommand\dY{\mathcal{Y}}
\newcommand\rH{{\rm H}}
\newcommand\real{{\mathbb R}}
\newcommand{\bZ}{{\mathbb Z}}
\newcommand{\bZO}{\bZ\sb{\otimes}}
\newcommand{\hZO}{\widehat{\bZ}\sb{\otimes}}

\newcommand{\bN}{{\mathbb N}}
\newcommand{\bF}{\mathbb F}
\newcommand{\Fp}{\bF\sb{p}}

\newcommand{\Aut}{\operatorname{Aut}}
\newcommand{\colim}{\operatorname{colim}}
\newcommand{\Hom}{\operatorname{Hom}}
\newcommand{\Map}{\operatorname{Map}}
\newcommand{\Mor}{\operatorname{Mor}}
\newcommand{\sk}[1]{\operatorname{sk}\sb{#1}}
\newcommand{\Obj}{\operatorname{Obj}\,}
\newcommand{\stt}{\star}
\newcommand{\fd}{f\sbu}

\newcommand{\gd}{g\sbu}

\newcommand{\op}{\sp{\operatorname{op}}}
\newcommand{\aC}{{\mathbf a}}
\newcommand{\haC}{\widehat{\aC}}
\newcommand{\cA}{{\mathbf A}}
\newcommand{\Ab}{{\mathbf{Ab}}}
\newcommand{\cC}{{\mathbf C}}
\newcommand{\gC}{{\mathbf g}}
\newcommand{\cG}{{\mathbf G}}
\newcommand{\Set}{{\mathbf{Set}}}
\newcommand{\Seta}{\Set\sb{\ast}}
\newcommand{\CS}{\Set\sp{\squa\op}}

\newcommand{\CSa}{\Seta\sp{\squa\op}}
\newcommand{\LCSa}{\Seta\sp{\LBox\op}}
\newcommand{\Sp}{{\mathbf{Spec}}}
\newcommand{\squa}{\pmb \square}
\newcommand{\LBox}{\overline{\squa}}
\newcommand{\cT}{{\mathbf T}}
\newcommand{\tria}{\pmb \triangle}
\newcommand{\nul}[1]{\operatorname{nul}\sb{#1}}
\newcommand{\nnC}{\nul{n}\cC}
\newcommand{\Nul}[1]{\operatorname{Nul}\sb{#1}}
\newcommand{\NnC}{\Nul{n}\cC}
\newcommand{\NlnC}{\Nul{\leq n}\cC}
\newcommand{\Tln}{\cT(\leq n)}
\DeclareMathOperator\Mon{Mon}
\DeclareMathOperator\Ext{Ext}
\newcommand{\II}[2]{I\sp{#1}\sb{#2}}

\newcommand{\wh}{~--~}
\newcommand{\wwh}{--~}
\newcommand{\w}[2][ ]{\ \ensuremath{#2}{#1}\ }
\newcommand{\ww}[1]{\ \ensuremath{#1}}
\newcommand{\www}[2][ ]{\ensuremath{#2}{#1}\ }
\newcommand{\wwb}[1]{\ \ensuremath{(#1)}-}
\newcommand{\wb}[2][ ]{\ (\ensuremath{#2}){#1}\ }
\newcommand{\wbb}[2][ ]{(\ensuremath{#2}){#1}\ }
\newcommand{\wref}[2][ ]{\ \eqref{#2}{#1}\ }
\newcommand{\id}{\operatorname{Id}}
\newcommand{\inc}{\operatorname{inc}}
\newcommand{\image}{\operatorname{Im}}
\newcommand{\Ker}{\operatorname{Ker}}
\begin{document}
%
%
\title{Higher order derived functors and the Adams spectral sequence}
\author[H.-J.~Baues]{Hans-Joachim Baues}
\address{Max-Planck-Institut f\"{u}r Mathematik\\
Vivatsgasse 7\\ 53111 Bonn, Germany}
\email{baues@mpim-bonn.mpg.de}
\author[D.~Blanc]{David Blanc}
\address{Department of Mathematics\\ University of Haifa\\ 3498838 Haifa, Israel}
\email{blanc@math.haifa.ac.il}
\date{\today}

\begin{abstract}
Classical homological algebra studies chain complexes, resolutions, and derived
functors in additive categories. In this paper we define \emph{higher order}
chain complexes, resolutions, and derived functors in the context of a new
type of algebraic structure, called an \emph{algebra of left cubical balls}.
We show that higher order resolutions exist in these algebras, and that they
determine higher order \ww{\Ext}-groups. In particular, the $E\sb{m}$-term of
the Adams spectral sequence \wb{m> 2} is such a higher \ww{\Ext}-group,
providing a new way of constructing its differentials.

\end{abstract}
\maketitle

\setcounter{section}{0}

%
%
\section*{Introduction}
\label{cint}

Topologists have been working on the problem of calculating the homotopy
groups of spheres for around eighty years, and many methods have been developed
for this purpose. One of the most useful tools for this purpose is the Adams
spectral sequence \w[,]{E\sb{2},E\sb{3},E\sb{4},\dotsc} converging to the 
$p$-completed stable homotopy groups of the sphere. Adams computed the 
\ww{E\sb{2}}-term of the spectral sequence, and showed that it is 
algebraically determined:
\begin{myeq}
\label{eqtwoterm}
E\sb{2}\sp{s,t}~\cong~\Ext\sp{s,t}\sb{\dA}(\Fp,\Fp)~,
\end{myeq}
where the derived functor \w{\Ext} is taken for modules over the mod $p$
Steenrod algebra $\dA$ of primary mod $p$ cohomology operations (cf.\ \cite{AdSS}).

Elements in the mod $p$\ cohomology \w{H\sp{n}(X;\Fp)} of a space $X$ are given
by homo\-to\-py classes of maps from $X$ to an Eilenberg-Mac~Lane space
\w[,]{K(\Fp,n)}  and the elements of the Steenrod algebra are homotopy
classes of maps between such Eilenberg-Mac~Lane spaces, acting on cohomology
classes by composition.

This structure suffices to recover the \ww{E\sb{2}}-term of the Adams spectral
sequence, by \wref[.]{eqtwoterm} However, in order to determine the higher terms
\w[,]{E\sb{3},E\sb{4},\dotsc} we have to look not only at homotopy classes of maps,
but at the \emph{space} \w{\Map(X,K(\Fp,n))} of all maps from $X$ to
\w[,]{K(\Fp,n)}  together with the action of the mapping spaces
\w{\Map(K(\Fp,n),K(\Fp,k))} on it. This structure is called the
\emph{Eilenberg-Mac~Lane mapping algebra} of $X$ (cf.\ \cite{BBlaC}).

As we show in this paper, this mapping algebra suffices to compute the
whole Adams spectral sequence (cf.\ \cite{BBlaS}). However, this structure
is still too complicated for computational purposes, because it involves
the \emph{topology} of the mapping spaces, and is not algebraic in nature.
Therefore, our main aim here is to extract from these spaces the appropriate
algebraic data needed to calculate all the differentials.

For this purpose, we consider singular cubes \w{\sigma:I\sp{n}\to\Map(X,Y)} in the
above mapping spaces. Appropriate collections of such cubes can be glued together
to form a \emph{left cubical ball}, as in Figure \ref{fig1} (see \S \ref{clcb}
below). Such balls have a new kind of combinatorial-algebraic structure
(coming from pasting operations, and so on), which we call an
\emph{algebra of left $n$-cubical balls}. The case \w{n=1} is discussed in
detail in Section \ref{stcaone}, where we show that an algebra of left
$1$-cubical balls corresponds to the notion of an abelian track category.

%
%
\begin{figure}[htbp]
\begin{center}
\begin{picture}(200,180)(20,10)
%
%
\put(70,100){\line(1,0){100}}
\put(120,50){\line(0,1){100}}
\put(80,60){\line(1,1){80}}
\put(30,60){\line(1,0){50}}
\put(30,60){\line(1,1){40}}
\put(80,140){\line(1,-1){80}}
\put(30,140){\line(1,0){50}}
\put(30,140){\line(1,-1){40}}
\put(80,140){\line(0,1){50}}
\put(80,190){\line(1,-1){40}}
\put(120,150){\line(1,1){40}}
\put(160,190){\line(0,-1){50}}
\put(160,140){\line(1,0){50}}
\put(170,100){\line(1,1){40}}
\put(170,100){\line(1,-1){40}}
\put(160,60){\line(1,0){50}}
\put(160,60){\line(0,-1){50}}
\put(120,50){\line(1,-1){40}}
\put(80,10){\line(1,1){40}}
\put(80,10){\line(0,1){50}}
\put(120,100){\circle*{5}}
\put(130,92){{\scriptsize $0$}}
\end{picture}
\end{center}
\caption[fig1]{A left $2$-cubical ball}
\label{fig1}
\end{figure}
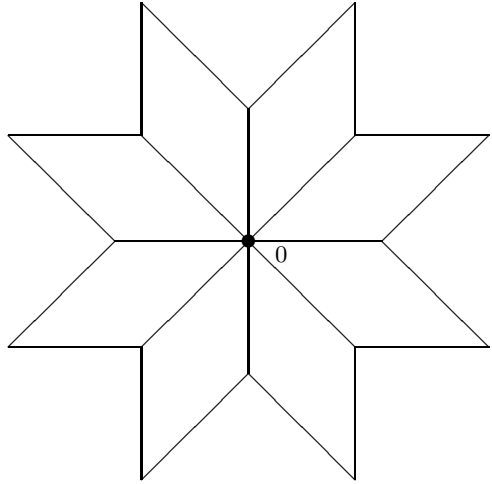

In such algebras, we define the notion of a higher order chain complex, replacing
the equation \w{\partial\sb{n-1}\circ\partial\sb{n}=0} in an ordinary chain 
complex by nullhomotopies \w[,]{H\sb{n}:\partial\sb{n-1}\circ\partial\sb{n}\sim 0} 
and higher nullhomotopies (see Section \ref{cccz}). This allows us to define 
higher order resolutions, in Section \ref{cehor}, and show:

%
%
\begin{thma}
Higher order resolutions exist in any algebra of left $n$-cubical balls.
\end{thma}
See Theorem \ref{thrt} below\vsm .

It turns out that such higher order resolutions provide a method for calculating 
the higher terms of the Adams spectral sequence, which thus may be
identified as certain $m$-th order derived functors, called $m$-th order
\ww{\Ext}-groups:

%
%
\begin{thmb}
Any $n$-th order resolution determines the higher order
\ww{\Ext}-groups \w{E\sb{m}} for \w[.]{m\leq n+2}
In the algebra of left cubical balls determined by the Eilenberg-Mac~Lane 
mapping algebra, these higher order \ww{\Ext}-groups compute the 
\ww{E\sb{m}}-terms of the Adams spectral sequence for \w[.]{m\leq n+2}
\end{thmb}
See Theorems \ref{twelldef} and \ref{temma} below.

\begin{remark}
As shown in \cite{BJSe}, at the secondary level (that is, computing the 
\ww{d\sb{2}}-differential), this algebra of cubical balls can be replaced 
by an ordinary differential graded algebra. In \cite{BFran}, the first author 
and Martin Frankland show that this can also be done at the tertiary level. 
It is conjectured in \cite{BauHO} that there is a single differential graded 
algebra, extracted from the algebra of cubical balls, from which the whole 
Adams spectral sequence can be computed.
\end{remark}

\begin{remark}
The associative composition between mapping spaces induces an associative 
operation $\otimes$ on singular cubes, because the product of cubes is again 
a cube.  This property does not hold for simplices, which is why we chose 
such singular cubes as our main technical tool here. In \cite{BFran}, too, 
the $\otimes$-product of cubes plays a central role. This is because the 
geometry of cubes is better suited to describing composition of (higher) 
homotopies than the geometry of simplices.

We observe that, in the same way, cubes (rather than simplices) are used 
in the literature to describe higher order homotopy operations, such as 
Toda brackets (see \cite{BJTurH}) and the higher homotopy commutativity 
of diagrams (see \cite{BVogHI}).
\end{remark}

\begin{remark}
Note that the Steenrod squares, represented by Steenrod's original
\ww{\cup\sb{i}}-operations at the chain level, generate the whole Steenrod algebra
$\dA$. In \cite{KristS,KristCF}, Kristensen used these \ww{\cup\sb{i}}-operations
to calculate secondary cohomology operations. Similarly, the algebra of left
cubical balls associated to the Eilenberg-Mac~Lane mapping algebra can be
used to extend Kristensen's approach to higher dimensions.
\end{remark}

\begin{remark}
This paper is framed in terms of the Eilenberg-Mac~Lane mapping algebra, since
the most important application of our methods is to the classical Adams
spectral sequence.  However, our methods may be extended to other mapping
algebras derived from other (topologically-enriched) Quillen model categories.
\end{remark}

\begin{conv}
\label{ctop}
The category of compactly generated Hausdorff spaces (cf.\ \cite{SteCC}, and
compare \cite{VogtCC}) is denoted by \w[,]{\TT} and that of pointed connected
compactly generated spaces by \w[.]{\Ta} In particular, we shall assume that
all (pointed) topological mapping spaces in this paper take value in
\w{(\TT,\times)}
(respectively, \w[).]{(\Ta,\wedge)} See also \cite[\S 1.1]{GWhitH}.
\end{conv}

\begin{ack}
The authors wish to thank the referee for his or her careful reading of the paper
and many pertinent remarks.
The first author would like to thank the Department of Mathematics at the
University of Haifa, and the second author would like to thank the Max Planck
Institut f\"{u}r Mathematik in Bonn, for their hospitality while this research
was being carried out. The second author was partly supported by Israel Science
Foundation Grant 247377/11.
\end{ack}

%
%
\sect{Left cubical sets}
\label{clcs}

A nullhomotopy is a homotopy between some map and the zero map between
pointed topological spaces. We wish to consider higher homotopies, which are
described by singular cubes in mapping spaces; in particular, we have higher
nullhomotopies, which lead us to the concept of \emph{left cubical sets},
defined in this section. The main examples we shall need are the left
$n$-cubical sets \w{\nul{n}(X)} and \w{\Nul{n}(X)} of a pointed space $X$.

We  first recall some basic properties of cubical sets:
let \w{I = [0, 1]} be the unit interval and let
\w{I\sp{n} = I\times\dots \times I}
be the $n$-dimensional cube. We have inclusions
\w{d\sb{\vare}\sp{i}:I\sp{n-1}=I\sp{i-1}\times\{\vare\}\times 
I\sp{n-i}\subset I\sp{n}}
for \w{1 \leq i \leq n} and \w[.]{\vare \in \{0, 1\}}  Here \w{I\sp{0}} is a
single point.

Let $\squa$ denote the category whose objects are cubes \w{I\sp{n}}
\wb[,]{n \geq 0} and whose morphisms are generated by \w{d\sb{\vare}\sp{i}} and
the projections \w[.]{s\sp{i}:I\sp{n}\to I\sp{n -1}}

A \emph{pointed cubical set} is a functor \w[,]{K:\squa\op \lra\Seta}
where \w{\Seta} is the category of pointed sets. As usual, \w{K(I\sp{n})} is
denoted \w{K\sb{n}} and \w{\ast\in K\sb{n}} is the base point. We write
\w{\dim(a) = n} if \w[.]{a \in K\sb{n}} See \cite{CisiP}, \cite{JardCH},
or \cite{IsaacS} for further details on the category of cubical sets.

\begin{defn}
\label{dlcs}
Let $\LBox$ be the subcategory of $\squa$ consisting of objects \w{I\sp{n}}
\wb{n \geq 0} and morphisms generated by \w[.]{d\sb{0}\sp{i}}
A \emph{left cubical set} is a functor \w[.]{\LBox\op\to\Seta} We write
\w{\partial\sp{i}} for \w{(d\sb{0}\sp{i})\sp{\ast}:K\sb{n}\to K\sb{n - 1}}
\wb[.]{1 \leq i \leq n} We also consider the full subcategories
\w{\LBox\sb{n}\subset \LBox} consisting of objects \w{I\sp{m}}
\wb[.]{0\leq m\leq n}
A functor \w{\LBox\sb{n}\op \lra\Seta} is called a \emph{left $n$-cubical set}.
\end{defn}

\begin{example}
\label{egnul}
Given a pointed cubical set $K$, one obtains a left cubical set \w{\nul{}(K)}
by setting
$$
\nul{}(K)\sb{m}~:=~\{a \in K\sb{m}~|\ \hsm(d\sb{1}\sp{i})\sp{\ast}
a = \ast\hsm \text{for}\hsm 1 \leq i \leq m\}.
$$
\noindent Accordingly, we define the left $n$-cubical set \w{\nul{n}(K)} to 
be the restriction of \w{\nul{}(K)} to \w[.]{\LBox\sb{n}\op}
\end{example}

\begin{remark}
\label{radjoint}
The following properties of (left) cubical sets are included here for future 
reference. Observe that \w{\nul{}} is a functor from pointed cubical sets
to left cubical sets. Its left adjoint \w{\dL:\LCSa\to\CSa} may be thought
of as a ``universal enveloping cubical set'' functor, described as follows:
given a left cubical set $M$, the pointed cubical set \w{\dL(M)} has one
$n$-cube \w{I\sp{n}\sb{a}} for each left $n$-cube \w[,]{a\in M} with
\w{(d\sb{1}\sp{i})\sp{\ast}I\sp{n}\sb{a}=\ast} (the base point) for each
\w[.]{1\leq i\leq n}

In addition, there is a degenerate \wwb{n+k}cube:
\begin{myeq}
\label{edegcube}
(s\sp{j\sb{1}})\sp{\ast}\dotsc(s\sp{j\sb{k}})\sp{\ast}I\sp{n}\sb{a} \hsm
\text{in}\hsm \dL(M)\hsm \text{for each iterated projection}\hsm
s\sp{j\sb{k}}\dotsc s\sp{j\sb{1}}:I\sp{n+k}\to I\sp{n}\text{~in~}\squa
\end{myeq}
\noindent (with identifications according to the cubical identities).

It is readily verified that \w{\dL(M)} is indeed a pointed cubical set,
with a natural isomorphism:
\begin{myeq}\label{eqadjoint}
\Hom\sb{\LCSa}(M,\nul{}(K))~\xra{\cong}~\Hom\sb{\CSa}(\dL(M),K)
\end{myeq}
\noindent for \w{K\in\CSa} and \w[.]{M\in\LCSa}
\noindent Moreover, both functors preserve dimensions of all cubes, so they
commute with the $n$-skeleton functor, yielding a left adjoint \w{\dL\sb{n}}
to \w[.]{\nul{n}}

For any cubical set $K$, let \w{\dC\sb{K}} be the partially ordered set of
all $k$-cubes \wb{k\geq 0} of $K$, ordered under inclusion. We have
\w[,]{K\cong\colim\sb{I\sp{k}\in\dC\sb{K}}\ I\sp{k}} where each \w{I\sp{k}} 
is thought of as a cubical set. We use this to define a monoidal structure on
\w[,]{\CS} given by:
\begin{myeq}\label{eqotimes}
K\otimes L~:=~\colim\sb{I\sp{j}\in\dC\sb{K},~I\sp{k}\in\dC\sb{L}}~I\sp{j+k}
\end{myeq}
\noindent (see  \cite[\S 3]{JardCH}). If $K$ and $L$ are pointed, there
is a cubical \emph{smash} functor
\begin{myeq}\label{eqsmash}
K\sotimes L~:=~(K\otimes L)~/~(\{\ast\}\otimes L \amalg K\otimes\{\ast\})
\end{myeq}
\noindent on \w[,]{\CSa} which is also defined on \w[.]{\LCSa} Moreover,
\w{\nul{}} and $\dL$ are monoidal with respect to $\sotimes$ on \w{\CSa}
and \w[,]{\LCSa} respectively.
\end{remark}

\begin{defn}
\label{dcsingular}
For a pointed space \w[,]{(X,\ast)} let \w{\SB X} be the \emph{singular}
pointed cubical set:  thus \w{(\SB X)\sb{n}} is the set of all maps
\w[,]{I\sp{n} \lra X} with base point \w[.]{o:I\sp{n} \lra \{\ast\} \subset X}

Thus, the set \w{\nul{}(X):=\nul{}(\SB X)} of all \emph{singular null cubes}
in $X$ consists of all maps \w{a: I\sp{n}\to X} \wb{n\geq 0} with
\w{a d\sb{1}\sp{i} = o} for \w[.]{1 \leq i \leq n} We let
\w[.]{\nul{n}(X):= \nul{n}(\SB X)}
\end{defn}

\begin{defn}
For any pointed topological space \w[,]{(X,\ast)} the
left $n$-cubical set \w{\Nul{n}(X)} is defined by
$$
\Nul{n}(X)\sb{m}:=
\begin{cases}\nul{}(X)\sb{m} & \text{for}\hsm m < n,\\ \nul{}(X)\sb{n}/\simeq &
\text{for}\hsm n = m.
\end{cases}
$$
\noindent Here we set \w{a \simeq b} for \w{a, b \in \nul{}(X)\sb{n}} if the maps
\w{a, b:I\sp{n} \lra X} are homotopic relative to the boundary
\w{\partial I\sp{n}}
of the cube \w[.]{I\sp{n}} Let \w{\{a\}} be the equivalence class of $a$;
we call \w{\{a\}} an \emph{$n$-track} in $X$.
\end{defn}

There is a surjective map of left $n$-cubical sets:
\begin{myeq}\label{eqm}
\nul{n}(X)~\epic~\Nul{n}(X)~,
\end{myeq}
which is the identity in dimension $< n$ and which carries $a$ with
\w{\dim(a) = n} to the $n$-track \w[.]{\{a\}} We point out that the
left $n$-cubical set \w{\Nul{n}(X)} is not the restriction of a cubical set.

\begin{remark}
Let $\overline{\tria}$ be the category with sets
\w{\{1, 2, \dotsc, n\}} \wb{n \geq 0} as objects, and order preserving
injective maps as morphisms. There is an isomorphism of categories
\w{\overline{\tria}\cong\LBox} which carries \w{\{1, 2, \dotsc, n\}}
to \w{I\sp{n}} and carries
$$
\{1, 2, \dotsc, n-1\}\cong\{1, \dotsc, \widehat{i}, \dotsc, n\} \subset
\{1, \dotsc, n\}
$$
\noindent to \w[.]{d\sb{0}\sp{i}} Here $\widehat{i}$ indicates that we omit $i$.
\end{remark}

%
%
\sect{$n$-graded categories enriched in left cubical sets}
\label{cncatlcs}

In this section we define the $\otimes$-composition of singular cubes in
a topologically enriched category, using the fact that the product of two cubes
is a cube.  It is important to note that (in a category enriched in
\w{(\Ta,\wedge)} \wwh see \S \ref{ctop}) the $\otimes$-composition of two
singular null cubes is a singular null cube.
The $\otimes$-composition respects the obvious grading by dimension,
yielding a graded category:

\begin{defn}\label{dgradec}
A \emph{graded set} is a sequence of sets \w[,]{K=(K\sb{n})\sb{n=0}\sp{\infty}} 
and a map of graded sets \w{f:K\to L} is a sequence of maps 
\w{f\sb{n}:K\sb{n}\to L\sb{n}}
\wb[.]{n \geq 0} We write \w{\dim(x) = n} if \w[.]{x \in K\sb{n}}
A \emph{graded category} \w{\cG} is a category in which each morphism $f$
has a dimension \w{\dim(f) \geq 0} such that the composition \w{f g} satisfies
$$
\dim(f g) = \dim(f) + \dim(g).
$$
\noindent Thus, all morphism sets \w{\Mor\sb{\cG}(X, Y)} are graded sets.

An \emph{$n$-set} $L$ is a finite sequence of sets \w[.]{(L\sb{0},\dotsc,L\sb{n})}
For example, the \emph{$n$-skeleton} \w{(K\sb{0}, \dotsc, K\sb{n})} of a graded set
is an $n$-set. An \emph{$n$-graded category} consists of morphism sets which \
are $n$-sets and composition \w{f g} is defined if \w[.]{\dim(f)+\dim(g)\leq n}
For example, the $n$-skeleton of a graded category is an $n$-graded category.

An \emph{$n$-graded category enriched in left cubical sets} is an $n$-graded
category such that morphism sets \w{\Mor(X, Y)} are left $n$-cubical sets
with operators \w{(d\sb{0}\sp{i})\sp{\ast} = \partial\sp{i}} satisfying
\begin{myeq}\label{eqfacemaps}
\partial\sp{i} (f g) = \begin{cases}
(\partial\sp{i} f) g        & \hsm \text{for}\hsm i \leq \dim(f)\\
f(\partial\sp{i-\dim(f)} g) & \hsm \text{for}\hsm i > \dim(f)
\end{cases}
\end{myeq}
\noindent Moreover, the zero maps \w{o\sp{n} \in \Mor(X, Y)\sb{n}} 
\wb{n \geq 0} satisfy
$$
o\sp{n} g = o\sp{n + \dim(g)} \hsm \text{and}\hsm f o\sp{m} = o\sp{\dim(f) + m}.
$$
\end{defn}

\begin{example}
\label{egnull}
Let $\cC$ be a category enriched in \w{(\Ta,\wedge)} where $\wedge$ is the
smash product of pointed topological spaces (cf.\ \S \ref{ctop}). Thus for
every \w[,]{X,Y\in\Obj(\cC)}  there is a zero map
\w[,]{o\in\Mor\sb{\cC}(X, Y)} satisfying \w{o g = o} and \w{f o = o} for any
\w[.]{f, g \in \Mor\sb{\cC}}

The $\otimes$-composition of two singular (null) cubes \w{f \otimes g} is
defined to be the composite
$$
\xymatrix@1{f \otimes g:I\sp{n} \times I\sp{m}\ar[r]^(.38){f \times g} &
\Mor\sb{\cC}(Y, X) \times \Mor\sb{\cC}(Z, Y)\ar[r]^(.63)\mu & \Mor\sb{\cC}(Z, X)}~,
$$
\noindent where $\mu$ is the composition in $\cC$.
We see that the $\otimes$-composition of two singular null cubes is a
singular null cube, so the category \w{\nul{}\cC} is a enriched in left
cubical sets, with \w[.]{\nul{}\cC(X,Y):=\nul{}\Mor\sb{\cC}(X,Y)}
\end{example}

\begin{defn}
\label{dnullcat}
For $\cC$ as above, the $n$-skeleton of \w[,]{\nul{}\cC} denoted by
\w[,]{\nnC} is given by the $n$-cubical sets \w[,]{\nul{n}\Mor\sb{\cC}(X, Y)}
and \w{\NnC} is defined similarly. There is a \emph{quotient functor}
$$
\nnC~\epic~\NnC~,
$$
\noindent given by the quotient maps:
$$
\xymatrix@1{\nul{n}\Mor\sb{\cC}(X, Y)\ar[r] & \Nul{n}\Mor\sb{\cC}(X, Y)}
$$
\noindent (see~\wref[).]{eqm}

Thus \w{\NnC} is an $n$-graded category with the composition defined by the
equivalence class \w{\{f \otimes g\}} for \w[.]{\dim(f) + \dim(g) = n}
The $n$-graded categories \w{\nnC} and \w{\NnC} are enriched in
left $n$-cubical sets.
\end{defn}

These two examples \w{\nnC} and \w{\NnC} are the reason for our definition
of the partially-defined composition in $n$-graded categories in
\S \ref{dgradec} above.

We also consider functors between $n$-graded categories, and in the next
section we will use such functors to define higher order chain complexes.

For \w[,]{n = 0} the ($0$-graded) category
$$
\Nul{0}\cC = \pi\sb{0} \cC
$$
\noindent has morphisms \w{X \lra Y} given by the path components of
\w[.]{\Mor\sb{\cC} (X, Y)}

%
%
\sect{The chain category \ww{\bZO}}
\label{cccz}

A chain complex in any pointed category $\M$ may be defined as a pointed
functor from a suitable indexing category.  A higher order chain complex is
a functor between $n$-graded categories, from a more elaborate indexing category
\w{\bZO\sp{n}} (which we now describe) into the category \w{\NnC}
of \S \ref{dnullcat}. The reader can skip the explicit definitions
of \w{\bZO} and \w{\bZO\sp{n}} in \S \ref{dmonn} and \S \ref{dbzo}; we only need
to know that they are well-defined.

\begin{defn}\label{dmonn}
Let \w{N~:=~\Mon(\stt, J)} be the free monoid generated by elements $\stt$ and
$J$, so the elements of $N$ are words in the letters $\stt$ and $J$.

Let \w{\deg, \dim:N \lra(\bN\sb{0}, +)} be monoid homomorphisms defined by
\begin{align*}
\deg(\stt) &= 1 & \deg(J) &= 1\\
\dim(\stt) &= 0 & \dim(J) &= 1.
\end{align*}

For example, \w{V = \stt \stt J \stt J J} is a word in $N$, with \w{\deg(V)=6}
the length of the word $V$, and \w{\dim(V) = 3} the number of letters $J$
in $V$. Let $\emptyset$ be the empty word, which is the unit in the monoid $N$.

We associate with $J$ the unit interval \w{I = [0, 1]} and with $\stt$ the
one point space \w[.]{\{0\}} For any word $V$, let $\overline V$ be the space
defined by
$$
\overline V =
\begin{cases}
I & \hsm\text{if}\hsm V = J\\
\{0\} & \hsm \text{if}\hsm V = \stt\\
\overline V\sb{1} \times \overline V\sb{2} & \hsm \text{if}\hsm V = 
V\sb{1} V\sb{2}~.
\end{cases}
$$

We say that $V$ is \emph{in the boundary of $W$} with \ww{V, W \in N}
\wb{V\neq W} if there is an inclusion \w[.]{\overline V \subset \overline W}
This implies \w{\deg V = \deg W} and \w[.]{\dim V \leq \dim W} By projecting
the spaces \w[,]{\{0\}}  one gets the homeomorphism
$$
\overline V \cong I\sp{\dim V}~.
$$

If $V$ is in the boundary of $W$, there is a unique inclusion \w{d\sb{V, W}}
of cubes in the category $\LBox$ (see Definition \ref{dlcs}) such that
$$
\xymatrix{
\overline V\ar[d]\ar[r]^(.4)\cong & I\sp{\dim V}\ar[d]\sp{d\sb{V, W}}\\
\overline W\ar[r]^(.4)\cong & I\sp{\dim W}
}
$$
\noindent commutes.

Now, consider elements $\stt$ and \w{I\sb{n}} \wbb[,]{n \geq 1} which generate
the monoid
$$
M~:=~\Mon(\stt, I\sb{n}~:\ n \geq 1)/I\sb{n} \circ I\sb{m} = I\sb{n + m}~.
$$
\noindent The multiplication in $M$ is denoted by $\circ$. Here,
\w{\Mon(\stt, I\sb{n}~:\ n \geq 1)} denotes the free monoid. In $M$, we divide
out by the relation \w{I\sb{n} \circ I\sb{m} = I\sb{n + m}} for \w[.]{n, m \geq 1}

There is a canonical isomorphism of monoids
$$
\xymatrix@1{M\ar[r]^(.48)\cong & N}
$$
\noindent which carries $\stt$ to  $\stt$ and \w{I\sb{n}} to the $n$-fold product
\w[.]{J\sp{n} = J \cdots J} Using this isomorphism, we obtain the functions
\w{\deg} and \w{\dim} on $M$.

We introduce on $M$ a further multiplication $\otimes$ defined by
\begin{myeq}\label{eqotime}
V \otimes W~=~V \circ \stt \circ W \hsp \text{for}\hsm V, W \in M~,
\end{myeq}
\noindent where the right hand side is the product of the elements $V$, $\stt$,
and $W$ in the monoid $M$. The operation $\otimes$ is associative, but it
has no unit. For the empty word \w[,]{\emptyset\in M} we get
$$
\emptyset\otimes\emptyset~=~\emptyset \circ \stt \circ \emptyset~=~\stt.
$$
\end{defn}

\begin{defn}\label{dbzo}
We define the \emph{chain category} \w{\bZO} to be the following graded category:
the objects in \w{\bZO} are the integers \w[.]{i, j,\dotsc \in \bZ}
In addition to the identities \w[,]{1\sb{i}} with \w[,]{\dim(1\sb{i}) = 0}
the morphisms in \w{\bZO} consist of
$$
\xymatrix@1{(i, V):i\hsm \ar[r]^(.37)V & \hsm i - \deg V - 1 = j}
$$
for all \w[.]{V \in M} The composition of \w{V:i \lra j} and
\w{W:j \lra j - \deg W - 1 = k} \wb[,]{W \in M} is defined
$$
\xymatrix@1{(i, W\otimes V):i\hsm \ar[r]^(.38){W \otimes V} &
\hsm i - \deg(W \otimes V) - 1 = k}~.
$$
\noindent Here we have \w[,]{\deg(W \otimes V) = \deg W + \deg V + 1} so
that the composition is well defined. We also omit $\otimes$ in the notation
of the composite.

More precisely, morphisms in \w{\bZO} are pairs \w[,]{(i, V)} where
\w[,]{i \in \bZ} \w[,]{V \in M} and $i$ is the source of the morphism
\w{(i, V)} (also written  \w[).]{V:i \lra j} The target $j$ satisfies
\w[.]{j = i - \deg V - 1}

See Section \ref{cwc} below for an alternative point of view.
\end{defn}

The category \w{\bZO} is graded by dimension of elements in $M$, with
\w[.]{\dim(W \otimes V) = \dim(W) + \dim(V)} The $n$-\emph{skeleton}
\w{\bZO\sp{n}}  of \w{\bZO} \wb{n \geq 0}is an $n$-graded category. The 
$0$-skeleton \w{\bZO\sp{0}} consists only of identities and of the morphisms
\w[,]{V:i \to i - \deg V - 1} where $V$ is a power of the element $\stt$ in $M$.

If \w[,]{V = \emptyset} then \w{\emptyset:i \lra i - 1} is in \w[.]{\bZO\sp{0}}
The composition is \w[,]{\emptyset \otimes \emptyset = \stt:i \lra i - 2}
and so on. We observe:

\begin{lemma}
\label{lgcc}
The category \w{\bZO} is freely generated by the morphisms
\w{(i,\emptyset):i\to i - 1} and \w{(i, I\sb{k}):i \lra i - k - 1} for
\w[.]{i \in \bZ} \w[.]{k \geq 1}
\end{lemma}

%
%
\sect{Higher order chain complexes}
\label{chocc}

We are now in a position to define higher order chain complexes.
In Section \ref{cehor}, we will introduce the notion of a higher order resolution,
which is a special kind of higher order chain complex.

\begin{defn}\label{dprech}
Given an $n$-graded category $\cT$ enriched in left $n$-cubical sets
(such as \w[),]{\cT = \NnC} we consider a functor of $n$-graded categories
$$
K:\bZO\sp{n}~\to~\cT
$$
\noindent which carries an object \w{i \in \bZ} to the object \w{K\sb{i}:= K(i)}
in $\cT$. We say that $K$ satisfies the \emph{inclusion property} if
the following holds:

Given morphisms \w{V, W:i \lra j} in \w{\bZO\sp{n}} such that $V$ is in the
boundary of $W$, then the induced morphisms \w{K(V)} and \w{K(W)} in $\cT$
satisfy the equation
\begin{myeq}\label{epcc}
K(V) = d\sb{V, W}\sp{\ast} K(W) \hs \text{in}\hsm \Mor_\cT(K\sb{i}, K\sb{j})~.
\end{myeq}
\noindent Here \w{d\sb{V, W}\sp{\ast}} is defined by the structure of
\w{\Mor_\cT(K\sb{i}, K\sb{j})} as a left cubical set.

A functor $K$ satisfying the inclusion property \wref{epcc} is called an
\emph{$n$-th order pre-chain complex} in \w[.]{\cT}
\end{defn}

\begin{defn}\label{dbasedon}
Let \w{N > M} and \w[.]{\bZ(N, M) = \{k \in \bZ, N \geq k \geq M\}}
Then we obtain the full subcategory
$$
\bZ(N, M)\sb{\otimes} \subset \bZO
$$
\noindent consisting of objects \w[.]{k \in \bZ(N, M)} We say that $K$ is
\emph{concentrated} in \w{\bZ(N, M)} if
\w{K:\bZ(N, M)\sb{\otimes}\sp{n} \lra \cT} is a functor of $n$-graded categories.

Assume given a quotient functor \w[,]{\cT\sp{0} \lra \cA} (that is, a full
functor which is the identity on objects) yielding induced morphisms
$$
\delta\sb{i} = K(i, \emptyset)\sb{\ast}:K\sb{i}~\to~ K\sb{i - 1}
$$
\noindent in $\cA$ for each \w[.]{i \in \bZ}  We then say that $K$ is
\emph{based on the diagram}
\begin{myeq}
K\sb{M}~\leftarrow~\dotsc~\leftarrow~K\sb{i-2}~\xleftarrow{\delta\sb{i-1}}~
K\sb{i-1}~\xleftarrow{\delta\sb{i}}~K\sb{i}~\leftarrow~\dotsc~\leftarrow~K\sb{N}
\end{myeq}
in the category $\cA$.
\end{defn}

\begin{defn}\label{dnocc}
Let $\cC$ be a category enriched in pointed spaces (with zero maps $o$).
For \w[,]{\cT = \NnC} we consider a functor $K$ with the inclusion property,
$$
\xymatrix@1{K:\bZO\sp{n}\ar[r] & \NnC}~.
$$
\noindent We have in \w{\bZO\sp{n}} the \wwb{n + 1}tuple of morphisms
\w[:]{i \lra i - n - 2}
\begin{myeq}\label{eqtom}
(i, \partial I\sb{n + 1}) =
\begin{cases}
(i, \emptyset \otimes I\sb{n}),\\ (i, I\sb{n} \otimes \emptyset),\\
(i, I\sb{r} \otimes I\sb{s}), & r + s = n, r \geq 1, s \geq 1
\end{cases}
\end{myeq}
which yields the \wwb{n + 1}tuple of $n$-tracks
$$
K(i, \partial I\sb{n + 1})~=~(K(i, \emptyset \otimes I\sb{n}),~K(i, I\sb{1} \otimes
I\sb{n - 1}),~\dotsc,~K(i, I\sb{n - 1} \otimes I\sb{1}),~K(i, I\sb{n} 
\otimes \emptyset))~.
$$
\noindent These tracks are represented by maps
\w[.]{I\sp{n} \lra \Mor\sb{\cC}(K\sb{i}, K\sb{i - n - 2})}
In fact, these $n$-tracks yield a map
$$
\xymatrix@1{\alpha:S\sp{n} \approx \partial(I\sp{n + 1})\ar[r] &
\Mor\sb{\cC}(K\sb{i}, K\sb{i - n - 2})}
$$
\noindent on the boundary of the \wwb{n + 1}cube, well-defined up to homotopy.
Hence, the map $\alpha$ determines an \emph{obstruction element}
\begin{myeq}\label{eqoe}
\OO K(i, \partial I\sb{n + 1}) \in \pi\sb{n} \Mor\sb{\cC}(K\sb{i}, K\sb{i - n - 2})~.
\end{myeq}

We write \w{D\sb{n}(K\sb{i}, K\sb{i - n - 2})} for
\w[,]{\pi\sb{n}\Mor\sb{\cC}(K\sb{i}, K\sb{i - n - 2})} since these form a 
natural system \w{D\sb{n}} (see \S \ref{dsma} below).

We say that $K$ is an \emph{$n$-th order chain complex} in \w{\NnC} if the 
obstruction elements \wref{eqoe} vanish for all $i$. This is the 
\emph{obstruction property} of $K$.
\end{defn}

In Section \ref{cntc} below, we determine the basic formulas satisfied by 
these obstruction elements.

\begin{defn}\label{dtoda}
For $\cC$ as above, let
\begin{myeq}\label{eqtdb}
\xymatrix@1{K\sb{0} & K\sb{1}\ar[l]\sb{\delta\sb{1}} & 
K\sb{2}\ar[l]\sb{\delta\sb{2}} & \dotsc\ar[l] &
K\sb{n + 2}\ar[l]\sb{\delta\sb{n + 2}}}
\end{myeq}
be a diagram in \w{\cA = \pi\sb{0}(\cC)} \wb[.]{n \geq 1}
Consider all functors
$$
K:\bZ(0, n + 2)\sb{\otimes}\sp{n}~\epic~\cT = \NnC
$$
\noindent based on the diagram \wref{eqtdb} which satisfy the inclusion property.
Each such functor yields an obstruction element
$$
\OO K(n + 2, \partial I\sb{n + 1}) \in D\sb{n}(K\sb{n + 2}, K\sb{0})~.
$$
The set of all these elements is the classical \emph{higher order Toda bracket}
$$
\langle\delta\sb{1}, \dotsc, \delta\sb{n + 2}\rangle 
\subset D\sb{n}(K\sb{n + 2}, K\sb{0})
$$
\noindent (see \cite{GWalkL}).
\end{defn}

The set can be empty. If there exists an $n$-th order chain complex $K$
based on  diagram \wref[,]{eqtdb} then of course
\w[,]{0 \in \langle\delta\sb{1}, \dotsc, \delta\sb{n + 2}\rangle} by the 
obstruction property of $K$.

%
%
\sect{The $W$-construction}
\label{cwc}

We now present an alternative description of higher order chain complexes
not actually needed for this paper, based on the classical bar construction
\w{W\K} going back to Boardman-Vogt (see \cite[\S 3]{BVogHI} and
\cite[\S 6]{BoaH}). This construction is a topologically-enriched ``cofibrant
replacement''  for any small category $\K$ (not necessarily unital), which
serves as the indexing category for lax versions of functors \w[.]{\K\to\TT}
A cubically enriched variant of \w{W\K} was defined in \cite[\S 3.1]{BJTurH}
and \cite[\S 3.4]{BBlaC}; we shall require the following pointed setting:

\begin{defn}\label{dwconst}
Let $\K$ be a small category enriched in \w{(\Seta,\wedge)} (with zero maps
$o$). The  \emph{pointed W-construction} on $\K$, denoted by \w[,]{\Ws\K} 
is the (non-unital) category enriched in \w{(\CSa,\sotimes)} with object set
\w{\Obj\K} defined as follows:

First, for every \w[,]{a,b\in\Obj\K} the underlying graded pointed category 
of \w{\Ws\K} has an (indecomposable) morphism ($n$-cube) \w{\II{n}{\fd}} in 
\w{\Ws\K(a,b)\sb{n}} associated to each composable sequence
\begin{myeq}\label{eqcompseq}
\w{\fd=(a=a\sb{n+1}\xra{f\sb{n+1}}a\sb{n}\xra{f\sb{n}}a\sb{n-1}\dotsc a\sb{1}
\xra{f\sb{1}}a\sb{0}=b)}
\end{myeq}
\noindent of length \w{n+1} in $\K$.
In addition, \w{\Ws\K(a,b)} has a degenerate \wwb{n+k}cube
\w{(s\sp{j\sb{1}})\sp{\ast}\dotsc(s\sp{j\sb{k}})\sp{\ast}\II{n}{\fd}} for 
each iterated projection 
\w{s\sp{j\sb{k}}\dotsc s\sp{j\sb{1}}:I\sp{n+k}\to I\sp{n}} in
$\squa$ (with identifications according to the cubical identities).
The zero map in degree $k$ is
\w[,]{\II{k}{o}:=(s\sb{k})\sp{\ast}\dotsc(s\sb{1})\sp{\ast}\II{0}{o}}
and we identify \w{\II{n}{\fd}} with \w{\II{n}{o}} whenever at least one 
of the maps \w{f\sb{1},\dotsc,f\sb{n+1}} is $o$. Then \w{\Ws\K} is freely 
generated as a graded category with zero maps by these cubes. Composition 
in the category \w{\Ws\K} is denoted by $\sotimes$.

The cubical structure is determined by the face maps of the non-degenerate 
indecomposable cubes \w{\II{n}{\fd}} and the cubical identities, as follows:
\begin{enumerate}
\renewcommand{\labelenumi}{(\alph{enumi})~}
\item The $i$-th $1$-face of \w{\II{n}{\fd}} is
\w{\II{n-1}{f\sb{1}\circ\dotsc\circ (f\sb{i}\cdot f\sb{i+1})\circ\dotsc 
f\sb{n+1}}}
\wwh that is, we carry out (in the category $\K$) the $i$-th composition in 
\w[.]{\fd}
\item The $i$-th $0$-face of \w{\II{n}{\fd}} is the composite
\w[.]{\II{i-1}{f\sb{1}\circ\dotsc\circ f\sb{i}}~\sotimes~
       \II{n-i}{f\sb{i+1}\circ\dotsc\circ f\sb{n+1}}}
\item The cubical structure on the composites \w{\II{j}{\fd}\sotimes\II{k}{\gd}}
is defined by \wref{eqotimes} (or \wref[).]{eqfacemaps}
\end{enumerate}
\end{defn}

\begin{defn}
\label{dgamma}
Let $\Gamma$ be the category enriched in \w{(\Seta,\wedge)} with object set
$\bZ$ and a single non-zero arrow \w{d\sb{k+1}:k+1\to k} for each
\w[,]{k\in\bZ} satisfying \w{d\sb{k}\circ d\sb{k+1}=o} for all $k$.
\end{defn}

\begin{prop}
Let $\M$ be a category enriched in cubical sets with zero maps. There is a 
one-to-one correspondence between pointed cubical functors \w{\Ws\Gamma\to\M}
and pre-chain complexes in \w[,]{\nul{}\M} which restricts to a one-to-one 
correspondence between pointed cubical functors \w{\sk{n}\Ws\Gamma\to\M} and
$n$-th order pre-chain complexes in \w[.]{\nul{n}\M}
\end{prop}

\begin{proof}
Since \w{\bZO} is a free graded category, by Lemma \ref{lgcc}, we can define 
a one-to-one functor of graded categories \w{\Phi:\bZO\to\Ws\Gamma} which 
is the identity on objects by setting \w{\Phi(i,\emptyset):=\II{0}{d\sb{i}}} 
and \w[,]{\Phi(i,I\sb{k}):=\II{k}{\fd}} for 
\w[.]{\fd:=(i\xra{d\sb{i}}i-1\to\dotsc i-k\xra{d\sb{i-k}}i-k-1)}

We can think of \w{\bZO} as a category \w{\hZO} enriched in \w{(\LCSa,\sotimes)}
by setting \w{d\sp{\ast}\sb{V,W}(W)=V} if \w[,]{V\subseteq W} and adding zero 
maps. Note that a functor \w{K:\bZO\to\nul{}\M} is a pre-chain complex if 
and only if it induces a pointed cubical functor \w[.]{\hat{K}:\hZO\to\nul{}\M}

The universal enveloping functor \w{\dL:\LCSa\to\CSa} of Remark \ref{radjoint} 
is monoidal with respect to $\sotimes$, so the adjunction \wref{eqadjoint}
extends to categories of enriched functors. Moreover, $\Phi$ induces a 
natural isomorphism of pointed cubical categories
\begin{myeq}
\label{eqwconst}
\dL(\hZO)~\cong~\Ws\Gamma~,
\end{myeq}
\noindent so left cubical functors \w{\hZO\to\nul{}\M} indeed correspond to 
pointed cubical functors \w[.]{\Ws\Gamma\to \M} Since this correspondence 
preserves the grading, the same is true for $n$-th order pre-chain complexes.
\end{proof}

%
%
\sect{Resolutions and derived functors}
\label{crdf}

We now recall the classical definitions of derived functors and resolutions
(which are ordinary chain complexes) in additive categories. Our higher 
order resolutions will be higher order chain complexes based on such 
ordinary resolutions (in the sense of \S \ref{dbasedon}).

\begin{defn}
Let $\cA$ be a category enriched in abelian groups, i.e., a preadditive
category.
We denote the morphism sets for objects $X$, $Y$ in $\cA$ by
$$
\Hom\sb{\cA}(X, Y) = \Mor\sb{\cA}(X, Y)~.
$$
\noindent This is an abelian group, and morphisms \w{f:X' \lra X} and
\w{g:Y' \lra \ Y} in $\cA$ induce homomorphisms \w{\Hom(f, Y)} and
\w[.]{\Hom(X, g)} Let $\aC$ be a full subcategory of $\cA$.

Let $Y$ be an object in $\cA$. An \emph{$\aC$-resolution} of $Y$ is a diagram
$$
A\sbu = (\xymatrix@1{\dotsc\ar[r]\sp{\delta\sb{2}} & 
A\sb{1}\ar[r]\sp{\delta\sb{1}} &
A\sb{0}\ar[r]\sp{\delta\sb{0}} & A\sb{-1}})
$$
\noindent in $\cA$ with \w{A\sb{-1} = Y} and \w{A\sb{i} \in \aC} for
\w[,]{i \geq 0} such that, for all objects $B$ in $\aC$, the induced diagram
\w{\Hom(B, A\sbu)} is an exact sequence of abelian groups; in particular,
\w{\Hom(B, \delta\sb{0})} is surjective.

An \emph{$\aC$-coresolution} of $X$ is a diagram
$$
A\ubu = (\xymatrix@1{A\sb{1}\ar[r]\sp{\delta\sb{1}} & 
A\sb{0}\ar[r]\sp{\delta\sb{0}} &
A\sb{-1}\ar[r]\sp{\delta\sb{-1}} & \dotsc})
$$
\noindent in $\cA$ with \w{A\sb{1} = X} and \w{A\sb{i} \in \aC} for
\w[,]{i \geq 0} such that for all objects $B$ in $\aC$ the induced diagram
\w{\Hom(A\ubu, B)} is an exact sequence of abelian groups. Here
\w{\Hom(\delta\sb{1}, B)} is surjective.
\end{defn}

The next result is proved in~\cite[1.3]{BJSe}:

\begin{lemma}
Suppose
\begin{enumerate}
\renewcommand{\labelenumi}{(\arabic{enumi})\ }
\item the coproduct of any family of objects of $\aC$ exists in $\cA$ and 
belongs to $\aC$ again,
\item there is a small subcategory $\gC$ of $\aC$ such that every object 
of $\aC$ is a retract of a coproduct of a family of objects from $\gC$,
\end{enumerate}
\noindent then every object of $\cA$ has an $\aC$-resolution.
\end{lemma}

The dual statement also holds:

\begin{lemma}
Suppose
\begin{enumerate}
\renewcommand{\labelenumi}{(\arabic{enumi})'\ }
\item the product of any family of objects of $\aC$ exists in $\cA$ and belongs 
to $\aC$ again,
\item there is a small subcategory $\gC$ of $\aC$ such that every object of 
$\aC$ is a retract of a product of a family of objects from $\gC$,
\end{enumerate}
then every object of $\cA$ has an $\aC$-coresolution.
\end{lemma}

One obtains~(1)' and (2)' by replacing the categories $\cA$ and $\aC$, 
respectively, in (1) and~(2) by the opposite categories \w{\cA\op} and 
\w[.]{\aC\op} Given a functor \w[,]{F:\cA \lra \dA} where $\dA$ is an 
abelian category and $F$ is linear (i.e., enriched in the category of abelian 
groups), its \emph{derived functors} are defined to be the homology 
(respectively, cohomology) groups
\begin{align*}
(L\sb{n} F)(X)& := \rH\sb{n} F(A\sbu)~,\\
(R\sp{n} F)(Y)& := \rH\sp{n} F(A\ubu)~.
\end{align*}
Here \w{A\sbu} is a resolution of $X$ and \w{A\ubu} 
a coresolution of $Y$.

We need the following concept of a $\Sigma$-algebra which allows the 
definition of a bigraded \ww{\Ext}-group.

\begin{defn}\label{dsa}
A \emph{$\Sigma$-algebra} \w{\cA = (\cA, \aC, \Sigma)} is an 
additive category $\cA$ together with an additive subcategory $\aC$ 
and an additive endofunctor \w{\Sigma:\cA\to\cA} of $\cA$ which 
carries $\aC$ to $\aC$ and which carries an $\aC$-resolution 
\w{A\sbu} of $X$ in $\cA$ to an $\aC$-resolution 
\w{\Sigma A\sbu} of \w{\Sigma X} in $\cA$. 

Dually, we define an \emph{$\Omega$-algebra} \w{\cA = (\cA, \aC, \Omega)} 
where $\Omega$ carries an $\aC$-coresolution of $X$ in $\cA$ to 
an $\aC$-coresolution of \w{\Omega X} in $\cA$.
\end{defn}

Given a $\Sigma$-algebra $\cA$ and objects $X$, $Y$ in $\cA$, 
we define the \emph{bigraded \ww{\Ext}-group} by the cohomology
\begin{myeq}
\label{eqbeg}
\begin{split}
E\sb{2}\sp{s, t} &= \Ext\sb{\cA}\sp{s}(\Sigma\sp{t} X, Y),\\
\notag &= \rH\sp{s} \Hom\sb{\cA}(\Sigma\sp{t} A\sbu, Y),\\
\notag &=
\Ker \Hom\sb{\cA}(\Sigma\sp{t} \delta\sb{s + 1}, Y)/
\image \Hom\sb{\cA}(\Sigma\sp{t} \delta\sb{s}, Y).
\end{split}
\end{myeq}
Here \w{\Sigma\sp{t} = \Sigma \circ \dotsc \circ \Sigma} is the $t$-fold 
composite of $\Sigma$. Such groups appear in the $E\sb{2}$-term of the 
Adams spectral sequence.

%
%
\sect{Mapping algebras}
\label{cma}

In this section we consider topological analogues of $\Sigma$-algebras and
$\Omega$-algebras of Definition \ref{dsa}, in order to provide a setting for
defining higher order resolutions, and thus higher order derived functors.

\begin{defn}\label{dsma}
Let $\cC$ be a category enriched in pointed spaces.
We call $\cC$ a \emph{$\Sigma$-mapping algebra} if the category
\w{\cA = \pi\sb{0} \cC} is a $\Sigma$-algebra and for each \w[,]{n \geq 1}
the bifunctor \w[,]{D\sb{n}:\cA\op \times \cA\to\Ab}
defined by \w[,]{D\sb{n}(X, Y):= \pi\sb{n} \Mor\sb{\cC}(X, Y)} is equipped with a
natural isomorphism:
$$
\tau\sb{\Sigma}:D\sb{n}(X, Y)~\xra{\cong}~\Hom\sb{\cA}(\Sigma\sp{n} X, Y)
$$
\noindent for \w{X$ in $\aC} and $Y$ in $\cA$. Here
\w{\Sigma\sp{n} = \Sigma \circ \dotsc \circ \Sigma} is the $n$-fold composite
of the endofunctor $\Sigma$ of $\aC$.

Dually, $\cC$ is an \emph{$\Omega$-mapping algebra} if the category
\w{\cA = \pi\sb{0} \cC} is an $\Omega$-algebra and for each \w[,]{n \geq 1}
\w{D\sb{n}(X, Y):= \pi\sb{n} \Mor\sb{\cC}(X, Y)} is equipped with a natural
isomorphism:
$$
\tau\sb{\Omega}:D\sb{n}(X, Y)~\xra{\cong}~\Hom\sb{\cA}(X, \Omega\sp{n} Y)
$$
\noindent for $X$ in $\cA$ and $Y$ in $\aC$. 

In the context of a pointed model category, such as \w[,]{\Ta} one should think of 
\w{\tau\sb{\Sigma}} and \w{\tau\sb{\Omega}} as being related by the adjunction
of $\Omega$ and $\Sigma$.
\end{defn}

\begin{defn}\label{dsmac}
A $\Sigma$-mapping algebra $\cC$ is \emph{complete} if the endofunctor
$\Sigma$ of $\cA = \pi\sb{0} \cC$ is induced by an endofunctor $\Sigma$ of
$\cC$ and we have a binatural transformation
$$
\tau\sb{\Sigma}:\Mor\sb{\cC}(\Sigma A, Y)~\to~\Omega \Mor\sb{\cC}(A, Y)~,
$$
where we use the topological loop space functor on pointed spaces. We require
the functor \w{\Sigma:\cC \lra \cC} to preserve zero maps and coproducts
in $\cC$.

Dually, an $\Omega$-mapping algebra $\cC$ is \emph{complete} if the
endofunctor $\Omega$ of \w{\cA = \pi\sb{0} \cC} is induced by an endofunctor
$\Omega$ of $\cC$ and we have a binatural transformation
$$
\tau\sb{\Omega}:\Mor\sb{\cC}(Y, \Omega A)~\to~\Omega \Mor\sb{\cC}(Y, A)~.
$$
\noindent Again we require that the functor \w{\Omega:\cC\to\cC} preserve
zero maps and products in $\cC$. Iteration of \w{\tau\sb{\Sigma}} (respectively,
\w[)]{\tau\sb{\Omega}} induces the isomorphisms \w{\tau\sb{\Sigma}} 
(respectively, \w[)]{\tau\sb{\Omega}} in Definition \ref{dsma}.
\end{defn}

\begin{example}
\label{egspec}
There are a number of different simplicial model categories of spectra,
including the simplicial spectra of \cite{BFrieH}, the $S$-modules of
\cite{EKMMayR}, the symmetric spectra of \cite{HSSmiS}, and the orthogonal
spectra of \cite{MMSShipM}. All of these are naturally  pointed;  in this and
later sections, we let \w{\Sp} be any category of spectra which is
enriched in pointed topological spaces (or simplicial sets), with pointed
function spaces
$$
\Mor\sb{\Sp}(X,Y)~=~\Map(X, Y)\hspace*{7mm}\text{for $X$ and $Y$ in~}\Sp~.
$$
\noindent We always assume that $X$ and $Y$ are both fibrant and cofibrant 
in our chosen model category.

Clearly zero maps \w{o:X\to Y} are defined in \w[.]{\Sp} Let $\dX$ be
a class of objects in \w{\Sp} closed under coproducts and the suspension
$\Sigma$: that is, for \w{A, A' \in \dX} we have
\w[.]{A \vee A', \Sigma A \in \dX} Then we have \w[,]{\{\dX\} \subset \Sp}
where \w{\{\dX\}} is the full subcategory in \w{\Sp} with objects in $\dX$.
In this case \w[,]{\cC = \Sp} with
\w[,]{\aC = \pi\sb{0}\{\dX\} \subset \cA = \pi\sb{0}\cC} is a complete
$\Sigma$-mapping algebra.

Dually, let $\dY$ be a class of objects in \w{\Sp} such that $\dY$ is
closed under products and loop functor $\Omega$, that is, for
\w[,]{B, B' \in \dY} we have \w[.]{B \times B',\Omega B \in \dY} Then
we have
$$
\{\dY\} \subset \Sp~,
$$
\noindent where \w{\{\dY\}} is the full subcategory in \w{\Sp} with
objects in $\dY$. Then \w{\cC = \Sp} with
\w{\aC = \pi\sb{0}\{\dY\} \subset \cA = \pi\sb{0}\cC} is a complete
$\Omega$-mapping algebra.
\end{example}

\begin{example}
Let $p$ be a prime and let \w{H = \rH\Fp} be the mod $p$  Eilenberg-Mac~Lane
spectrum. Let $\dY$ be given by all products
$$
\Omega\sp{n\sb{1}} H \times \Omega\sp{n\sb{2}} H \times \dotsc 
\times \Omega\sp{n\sb{k}} H
$$
\noindent with \w{n\sb{i} \geq 0} for \w{i = 1, \dotsc,k} \wb[.]{k \geq 0}
Then \w{\cC = \Sp} with \w{\aC = \pi\sb{0}\{\dY\}} is a complete
$\Omega$-mapping algebra, which we call the
\emph{Eilenberg-Mac~Lane mapping algebra}. This is used in the Adams 
spectral sequence.
\end{example}

\begin{remark}
In the examples of mapping algebras above the category
\w{\cC = \Sp} is very large. For computations, however, we consider only the 
mapping algebras \w{\cC'} which are generated by \w{\{\dX\}} 
(respectively, \w[)]{\{\dY\}} and two further objects $X$ and $Y$ in \w[.]{\Sp}
\end{remark}

%
%
\sect{Higher order resolutions}
\label{cehor}

We can use the definitions of Section \ref{cma} to define higher order
resolutions, and state our main results about them, which will be proved in
Section \ref{ctaprt}.

These higher order resolutions will be used in Section \ref{cext} to describe
the higher terms of the Adams spectral sequence, providing an alternative way
of setting up this spectral sequence, avoiding the classical topological
construction of \cite{AdSS} using a tower of (co)fibrations of spectra.

\begin{defn}
Let $\cC$ be a $\Sigma$-mapping algebra with
\w[.]{\aC \subset \cA = \pi\sb{0} \cC} An $n$-th order chain complex
$$
\xymatrix@1{K:\bZ(\infty, -1)\sb{\otimes}\sp{n}\ar[r] & \NnC}
$$
\noindent is called \emph{$n$-th order resolution} of $X$ in \w{\NnC}
if it is based on an $\aC$-resolution in $\cA$,
$$
A\sbu = (\xymatrix@1{\dotsc\ar[r]\sp{\delta\sb{2}} & 
A\sb{1}\ar[r]\sp{\delta\sb{1}} &
A\sb{0}\ar[r]\sp{\delta\sb{0}} & A\sb{-1}})~,
$$
of \w[.]{X = A\sb{-1}}
\end{defn}

\begin{rthm}\label{tresol}
Given an $\aC$-resolution \w{A\sbu} of $X$ in $\cA$,
there is an $n$-th order resolution $K$ of $X$ based on \w{A\sbu}
in \w{\NnC} \wb[.]{n \geq 1}
\end{rthm}

This will follow from the more general Theorem \ref{thrt} below.

\begin{remark}
The Theorem shows that, if `minimal' $\aC$-resolutions exist (as in the
case of the Adams spectral sequence), then also an $n$-th order minimal
resolution exists which is based on a minimal resolution in $\cA$. This is
of high importance for computations.
\end{remark}

\begin{defn}
Dually, let $\cC$ be a $\Omega$-mapping algebra with
\w[.]{\aC \subset \cA = \pi\sb{0} \cC} If an $n$-th order chain complex
$$
\xymatrix@1{L:\bZ(+1, -\infty)\sb{\otimes}\sp{n}\ar[r] & \NnC}
$$
\noindent is based on an $\aC$-coresolution in $\cA$
$$
A\ubu = (\xymatrix@1{A\sb{1}\ar[r]\sp{\delta\sb{1}} & 
A\sb{0}\ar[r]\sp{\delta\sb{0}} &
A\sb{-1}\ar[r]\sp{\delta\sb{-1}} & \dotsc})~,
$$
\noindent with \w[,]{A\sb{1} = Y} we say that $L$ is an
\emph{$n$-th order coresolution} of $X$ in \w[.]{\NnC}
\end{defn}

\begin{drthm}\label{tdresol}
Given an $\aC$-coresolution \w{A\ubu} of $Y$ in $\cA$, there is an $n$-th 
order coresolution $L$ of $Y$ based on \w{A\ubu} in \w{\NnC} \wb[.]{n \geq 1} 
\end{drthm}

This will follow from Theorem \ref{thdrt} below.

\begin{remark}
In view of Lemma~3.6~(a) in~\cite{BJSe}, a first-order resolution
in \w{\Nul{1}\cC} is a secondary resolution in the sense of~\cite{BJSe}.
\end{remark}

%
%
\sect{Left cubical balls}
\label{clcb}

The higher order resolutions of Theorems \ref{tresol} and \ref{tdresol} are built
of left cubes in a mapping algebra $\cC$. The main goal of this paper is to
elucidate those properties of left cubes which are actually needed for the
construction of these resolutions, which are used in turn to calculate the
differentials in the Adams spectral sequence.

The relevant properties will be described using the new notion of a
\emph{left cubical ball}, which we introduce in this section. Such balls serve
as a book-keeping device to describe the combinatorics of higher tracks, and
allow us to define certain associated obstructions in homotopy groups of
mapping spaces in $\cC$ in the next section.

\begin{defn}\label{dball}
A \emph{ball} of dimension $n$ is a finite regular CW-complex $B$ with a
subcomplex \w{\partial B} and a homeomorphism of pairs
$$
(E\sp{n}, S\sp{n -1}) \approx (B, \partial B)
$$
\noindent where \w{E\sp{n}} is the Euclidean ball. We do not assume that
$B$ has only one $n$-cell. Two balls $B$ and \w{B'} are \emph{equivalent}
if there is a cellular isomorphism \w[.]{B \approx B'} A ball $B$ is a union
$$
B~=~B\sb{1} \cup \dotsc \cup B\sb{k}
$$
\noindent of closed $n$-cells \w{B\sb{i}} in $B$. We say that $A$ is a
\emph{sub-ball} of $B$ if \w{A = B\sb{i\sb{1}} \cup \dotsc \cup B\sb{i\sb{t}}} for
\w{1 \leq i\sb{1} < \dotsc < i\sb{t} \leq k} is a ball and if for \w[,]{t < k}
the closure of the complement \w{B\setminus A} in $B$ is also a ball, denoted by
\w[,]{A\sb{B}} so that \w[.]{B = A \cup A\sb{B}} We also require that
\w{S:=A\cap A\sb{B}} be a ball of dimension \w[,]{n-1} contained in both
\w{\partial A} and \w[,]{\partial A\sb{B}} with the closures of
\w{\partial A\setminus S} in \w{\partial A} and of
\w{\partial A\sb{B}\setminus S} in \w{\partial A\sb{B}}
likewise balls of dimension \w[.]{n-1}

If $A$ is also a sub-ball of a ball $C$ with
\w[,]{S = A \cap A\sb{B} = A \cap A\sb{C}} then we obtain the
\emph{union of complements}
$$
A\sb{B}\cup A\sb{C}~=~A\sb{B}\cup\sb{s}A\sb{C}~,
$$
\noindent which is also a ball of dimension $n$, being the union of two
$n$-balls along \wwb{n-1}balls $S$ in their boundaries (identified with
opposite orientations on $S$).  See Figure \ref{funion}.

%
%
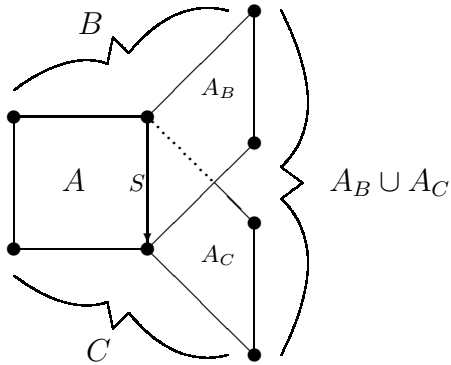
\begin{figure}[htbp]
\begin{picture}(100,140)(0,0)
%
%
\put(0,50){\circle*{5}}
\put(0,50){\line(1,0){50}}
\put(0,50){\line(0,1){50}}
\put(0,100){\circle*{5}}
\put(50,100){\vector(0,-1){48}}
\put(43,72){{\scriptsize $S$}}
\put(50,50){\circle*{5}}
\put(18,72){$A$}
\put(50,50){\line(1,1){40}}
\put(90,90){\circle*{5}}
\put(90,90){\line(0,1){50}}
\put(0,100){\line(1,0){50}}
\put(50,100){\circle*{5}}
\put(50,100){\line(1,1){40}}
\put(90,140){\circle*{5}}
\put(70,108){{\scriptsize $A\sb{B}$}}
\bezier{400}(0,110)(20,125)(35,120)
\bezier{40}(35,120)(36,125)(37,130)
\bezier{40}(37,130)(40,127)(43,124)
\bezier{400}(43,124)(60,145)(80,140)
\put(24,131){$B$}
\bezier{400}(0,40)(20,25)(35,30)
\bezier{40}(35,30)(36,25)(37,20)
\bezier{40}(37,20)(40,23)(43,26)
\bezier{400}(43,26)(60,5)(80,10)
\put(27,8){$C$}
\multiput(50,100)(2,-2){13}{\circle*{0.7}}
\put(75,75){\line(1,-1){15}}
\put(50,50){\line(1,-1){40}}
\put(90,10){\circle*{5}}
\put(90,10){\line(0,1){50}}
\put(90,60){\circle*{5}}
\put(70,45){{\scriptsize $A\sb{C}$}}
\bezier{400}(100,140)(120,100)(100,81)
\bezier{40}(108,75)(104,72)(100,69)
\bezier{40}(108,75)(104,78)(100,81)
\bezier{400}(100,10)(120,50)(100,69)
\put(118,72){$A\sb{B}\cup A\sb{C}$}
\end{picture}

\caption[funion]{Union of complements}
\label{funion}
\end{figure}
\end{defn}

\begin{example}\label{egtball}
Let \w{T\sb{0}\sp{n}} be the union of all cells
\w{I\sp{i - 1} \times\{0\}\times I\sp{n+1-i}}
in \w[,]{I\sp{n + 1}} and let \w{T\sb{1}\sp{n}} be the union of all cells
\w{I\sp{i - 1} \times \{1\} \times I\sp{n + i-1}} \wb[.]{i = 1, \dotsc, n + 1}
Thus \w{T\sb{\vare}\sp{n}} is the corner of \w{I\sp{n+1}} consisting of all
$n$-facets touching the vertex \w{(\vare,\vare,\dotsc,\vare)} \wb[,]{\vare=0,1}
and \w{T\sb{0}\sp{n}} and \w{T\sb{1}\sp{n}} are balls of dimension $n$, with
\w{n + 1} closed $n$-cells.
\end{example}

\begin{defn}\label{dlcb}
A \emph{left cubical ball} is a ball $B$ with a $0$-vertex
\w{\omi\in B - \partial B} equipped with the following:

\begin{enumerate}
\renewcommand{\labelenumi}{(\arabic{enumi})~}
\item  Each closed $n$-cell \w{B\sb{i}} has a homeomorphism
\w[.]{h\sb{i}:I\sp{n}\to B\sb{i}}
\item Each closed \wwb{n - 1}-cell $\ove$ has a homeomorphism
\w[.]{h\sb{e}:I\sp{n-1}\to\ove} such that if
\w[,]{\ove \subset B\sb{i} \cap B\sb{j}} the diagram
$$
\xymatrix{
B\sb{j} && ~\ove~\ar@{^{(}->}[rr] \ar@{_{(}->}[ll] && B\sb{i}\\
I\sp{n}\ar[u]\sp{h\sb{j}}_{\cong} &&
~I\sp{n-1}~ \ar@{_{(}->}[ll]\sp{d\sb{e, j}} \ar@{^{(}->}[rr]\sb{d\sb{e, i}}
\ar[u]_{\cong}^{h\sb{e}} && I\sp{n}\ar[u]^{\cong}\sb{h\sb{i}}
}
$$
\noindent commutes.
\end{enumerate}

Here \w{d\sb{e, j}} and \w{d\sb{e, i}} are morphisms in the
left cubical category $\LBox$. The vertex $\omi$ is also a vertex of each
\w{B\sb{i}} and the equivalence \w{h\sb{i}:I\sp{n} \approx B\sp{i}} carries
$\omi$ to $\omi$. Moreover, the union
$$
h\sb{1}(T\sb{1}\sp{n- 1}) \cup \dotsc \cup h\sb{k}(T\sb{1}\sp{n - 1}) =
\partial B
$$
\noindent is the boundary of $B$.
\end{defn}

\begin{example}
The $n$-ball \w{T\sb{0}\sp{n}} of Example \ref{egtball} is a left cubical
ball, as is the pushout of \w[,]{I\sp{n} \lla T\sb{0}\sp{n - 1} \lra I\sp{n}}
called the \emph{double} of \w[.]{I\sp{n}}
\end{example}

Further examples of left cubical balls of dimension $2$ appear in Figure \ref{fig1}
above and in Figure \ref{fig2}.

%
%
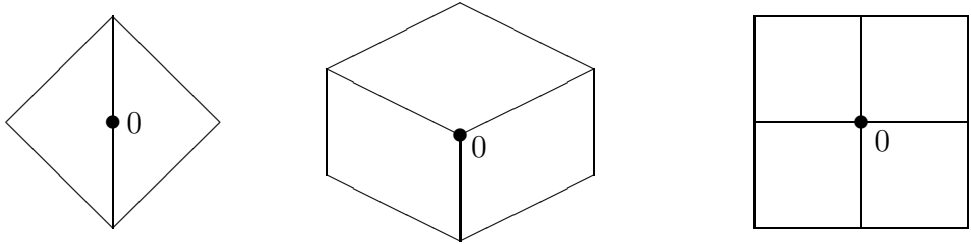
\begin{figure}[htbp]
\begin{picture}(370,100)(0,0)
%
%
\put(0,50){\line(1,1){40}}
\put(0,50){\line(1,-1){40}}
\put(40,90){\line(1,-1){40}}
\put(40,90){\line(0,-1){80}}
\put(40,10){\line(1,1){40}}
\put(40,50){\circle*{5}}
\put(45,46){$\omi$}
%
%
\put(120,70){\line(2,1){50}}
\put(120,70){\line(2,-1){50}}
\put(120,70){\line(0,-1){40}}
\put(170,45){\circle*{5}}
\put(174,37){$\omi$}
\put(170,45){\line(2,1){50}}
\put(170,45){\line(0,-1){40}}
\put(170,95){\line(2,-1){50}}
\put(220,70){\line(0,-1){40}}
\put(120,30){\line(2,-1){50}}
\put(170,5){\line(2,1){50}}
%
%
\put(280,10){\line(0,1){80}}
\put(320,10){\line(0,1){80}}
\put(360,10){\line(0,1){80}}
\put(280,10){\line(1,0){80}}
\put(280,50){\line(1,0){80}}
\put(280,90){\line(1,0){80}}
\put(320,50){\circle*{5}}
\put(325,39){$\omi$}
\end{picture}

\caption[fig2]{Some left cubical balls of dimension $2$}
\label{fig2}
\end{figure}

\begin{lemma}
Let $A$ be a sub-ball of $B$ and $C$, where $B$ and $C$ are left cubical; 
then the union of complements \w{A\sb{B} \cup A\sb{C}} is left cubical.
\end{lemma}

\begin{remark}
Let $B$ be a left cubical ball of dimension $n$ with $k$ closed $n$-cells.
Then $B$ is equivalent to the double of \w{I\sp{n}} for \w{k = 2} and $B$ is
equivalent to \w{T\sb{0}\sp{n}} for \w[.]{k = n + 1} For \w[,]{2 < k < n + 1} such
a ball does not exist. For \w[,]{k \geq n + 1} there is a one-to-one
correspondence between left cubical balls (up to equivalence) and simplicial 
complexes homeomorphic to the \wwb{n - 1}sphere \w[.]{S\sp{n - 1}} The
correspondence carries $B$ to the boundary of a small neighborhood of $0$
in $B$, where the link of the corner of an $n$-cube forms an \ww{n-1}-simplex.
Thus in Figure \ref{fig1}, the eight $2$-cubes yield an octagon (a
triangulation of the $1$-sphere).
\end{remark}

%
%
\sect{Obstructions}
\label{cob}

We now describe, for a general space $X$, the type of obstructions derived 
from cubical balls, which we shall later need in the case when $X$ is a 
mapping space.

\begin{defn}
Let $X$ be a pointed space with base point \w[,]{o\in X} let $B$ be a ball, and
let \w{a:B \lra X} be a map with \w[.]{a(\partial B) = o} We obtain the map
\begin{myeq}\label{eqobstr}
a:S\sp{n}~\approx~E\sp{n}/S\sp{n - 1}~\approx~B/\partial B~\longrightarrow~X~,
\end{myeq}
which represents an element \w{\OO(a) \in \pi\sb{n}(X)} in the $n$-th homotopy
group of $X$. Now let \w{B = B\sb{1} \cup \dotsc \cup B\sb{k}} be a left 
cubical ball. Then
$$
I\sp{n}~\xra{h\sb{i}}~B\sb{i}~\subset~B~\xra{a}~X
$$
\noindent is a left $n$-cube representing an $n$-track
\w[.]{a\sb{i} \in \Nul{n}(X)\sb{n}}

Then for \w{\overline e \subset B\sb{i} \cap B\sb{j}} we have the
\emph{gluing condition in $B$} (see Definition \ref{dlcb}).
\begin{myeq}\label{egc}
d\sb{e, i}\sp{\ast} a\sb{i} = d\sb{e, j}\sp{\ast} a\sb{j}.
\end{myeq}

A $k$-tuple \w{(a\sb{1}, \dotsc, a\sb{k})} of $n$-tracks \w{a\sb{i}} in
\w{\Nul{n}(X)\sb{n}} satisfying \wref{egc} is called a
\emph{left $n$-cubical ball}
in $X$.
\end{defn}

\begin{lemma}
A left $n$-cubical ball \w{(a\sb{1}, \dotsc, a\sb{k})} in $X$
yields a map \w{a:B \to X} with \w[,]{a(\partial B) = o}
well-defined up to homotopy relative to the boundary.
This defines the \emph{obstruction} 
\w{\OO\sb{B}(a\sb{1}, \dotsc, a\sb{k}) = \OO(a)} in
\w{\pi\sb{n}(X)} as above.\hfill$\Box$
\end{lemma}

Now let \w{B = T\sb{0}\sp{n} = B\sb{1} \cup \dotsc \cup B\sb{n + 1}} be the 
ball of \S \ref{egtball}, and let 
\w{a\sb{1}, \dotsc, a\sb{n + 1} \in\Nul{n}(X)\sb{n}} be
$n$-tracks satisfying the gluing condition \wref[.]{egc} Then we get
the \emph{boundary property}:

\begin{lemma}\label{lem:boundaryProperty}
\w{\OO\sb{T\sb{0}\sp{n}}(a\sb{1}, \dotsc, a\sb{n + 1}) = 0} if and only if 
there exist \w{\overline a \in \Nul{n + 1}(X)\sb{n + 1}} with 
\w{\partial\sp{i} \overline a} representing \w[.]{a\sb{i}}
\end{lemma}

\begin{proof}
We choose representatives \w{a'\sb{i}} of \w{a\sb{i}} which define a map
$$
\xymatrix@1{\overline{\overline a}:\partial I\sp{n + 1}\ar[r] & X}
$$
\noindent with \w{\overline{\overline a}(T\sb{1}\sp{n}) = 0} and
\w[.]{\overline{\overline a}|\sb{T\sb{0}\sp{n}} = 
a'\sb{1}\cup\dotsc\cup a'\sb{n + 1}}
Here \w{\overline{\overline a}} extends to \w{I\sp{n + 1}} if and only if
\w[.] {\OO(a'\sb{1} \cup \dotsc \cup a'\sb{n + 1}) = 0}
\end{proof}

The next result is the \emph{Complement Rule}.

\begin{lemma}\label{lcompl}
Let \w{B = A\sb{1} \cup \dotsc \cup A\sb{r}\cup B\sb{1}\cup\dotsc\cup B\sb{t}} 
and \w{C = A\sb{1} \cup \dotsc \cup A\sb{r}\cup C\sb{1}\cup\dotsc\cup C\sb{s}} 
be left cubical balls with a common sub-ball 
\w[.]{A = A\sb{1}\cup\dotsc\cup A\sb{r}} Then
$$
\OO\sb{C}(a\sb{1}, \dotsc, a\sb{r}, c\sb{1},\dotsc c\sb{s}) = 0
$$
\noindent implies that for \w{D = A\sb{B} \cup A\sb{C}} as in \S \ref{dball} 
we have:
$$
\hspace*{28mm}\OO\sb{B}(a\sb{1}, \dotsc, a\sb{r}, b\sb{1}, \dotsc, b\sb{t}) =
\OO_D(b\sb{1}, \dotsc, b\sb{t}, c\sb{1}, \dotsc, c\sb{s}).
$$
\end{lemma}

\begin{proof}
This follows from the Homotopy Addition Theorem (cf.\ \cite{MunkrS}) 
and our conventions on the union of complements (Definition \ref{dball}).
\end{proof}

\noindent Of course, there is the following \emph{Double Rule}:

\begin{lemma}
If \w{B = B\sb{1} \cup B\sb{2}} is the double of \w{I\sp{n}} then for
\w{a\sb{1} = a\sb{2}} we have:
$$
\hspace*{60mm}\OO\sb{B}(a\sb{1}, a\sb{2}) = 0~.\hspace*{60mm}\Box
$$
\end{lemma}

\begin{defn}
Let \w{B = B\sb{1} \cup \dotsc \cup B\sb{k}} be a left cubical ball. Then for each
\w{1 \leq i \leq k} we have a map
$$
\overline \vare\sb{i}:I\sp{n} \approx B\sb{i} \subset B \approx E\sp{n}~,
$$
\noindent where \w{I\sp{n}} and \w{E\sp{n}} are oriented by the inclusions
of \w{I\sp{n}} and \w{E\sp{n}} in \w[.]{\real\sp{n}} We set \w{\vare\sb{i} = +1}
if the map \w{\overline \vare\sb{i}} is orientation preserving, otherwise
\w[.]{\vare\sb{i} = -1} We call \w{\vare\sb{i}} the \emph{orientation sign}
of \w[.]{B\sb{i}}

Let \w{B = B\sb{1} \cup B\sb{2}} be the double of \w[.]{I\sp{n}} Then
\w[,]{\vare\sb{1} = -\vare\sb{2}} and we can choose \w{B\sb{1}} so that
\w[,]{\vare\sb{1} = 1} In this case we define the \emph{action} $+$ of
\w{\alpha \in \pi\sb{n}(X)} on an $n$-track \w{a \in \Nul{n}(X)\sb{n}} 
\wb{n \geq 1} by letting \w{a + \alpha} denote the unique $n$-track $b$ such that
\w[.]{\OO(b, a) = \alpha} This can of course be described explicitly in
terms of the coaction of the $n$-sphere on the $n$-ball.
\end{defn}

\begin{lemma}
The action $+$ yields a well defined effective and transitive action of the
group \w{\pi\sb{n}(X)} on the set of all $n$-tracks \w{a \in \Nul{n}(X)\sb{n}}
which coincide on the boundary (that is, \w[,]{\partial\sp{i} a = b\sb{i}} where
\w{(b\sb{1}, \dotsc, b\sb{n})} is fixed).\hfill $\Box$
\end{lemma}

\begin{lemma}
Let \w{B = B\sb{1} \cup \dotsc \cup B\sb{k}} be a left cubical ball and let
\w[,]{\OO\sb{B}(a\sb{1}, \dotsc, a\sb{k})} 
\w{\OO\sb{B}(a'\sb{1}, \dotsc, a'\sb{k})} be defined, where
$$
\begin{cases}
a'\sb{i} = a\sb{i} & \text{for $i \neq j$}\\
a'\sb{j} = a\sb{j} + \alpha & \text{for $i = j$, $\alpha \in \pi\sb{n}(X)$}.
\end{cases}
$$
\noindent Then we have the \emph{Action Formula}:
$$
\hspace*{40mm}\OO(a'\sb{1}, \dotsc, a'\sb{k}) =
\OO(a\sb{1}, \dotsc, a\sb{k}) + \vare\sb{j} \alpha~.\hspace*{40mm}\Box
$$
\end{lemma}

%
%
\sect{The algebra of left $n$-cubical balls}
\label{cntc}

In this section we collect together those properties of left $n$-cubical balls in
the mapping spaces of a mapping algebra $\cC$ which are needed for the
construction of higher order resolutions. These properties are encoded in
the notion of the \emph{algebra} \w{(\cT,\cA,D,\OO)} of left
$n$-cubical balls in $\cC$. The main example we have in mind is given as follows:

Let $\cC$ be a category enriched in pointed spaces. Let \w{n \geq 1} and let
\w{\cT~=~\NnC} and \w[.]{\cA~=~\pi\sb{0} \cC} We define a functor
\w{D:\cA\op \times \cA \lra \Ab} by
\w[,]{D(X, Y):= \pi\sb{n} \Mor\sb{\cC}(X, Y)} so
\w{\OO\sb{B}(a\sb{1}, \dotsc, a\sb{k})} is defined in \w{\Nul{n}\Mor\sb{\cC}(X, Y)}
as in \wref[.]{eqobstr}

For \w{n = 1} we assume that the group
\w{\pi\sb{1} \Mor\sb{\cC}(X, Y)} is abelian for all objects $X$, $Y$ in $\cC$.

As shown in Section \ref{cob}, the system \w{(\cT,\cA,D,\OO)} is a special case
of the following abstract concept:

\begin{defn}\label{dalgncb}
An \emph{algebra of left $n$-cubical balls} \wb{n \geq 1} is a
system \w{\cT=(\cT,\cA,D,\OO)}
given by an $n$-graded category $\cT$, a quotient functor
\w[;]{\cT\sp{0} \lra \cA} a bifunctor
\w[,]{D:\cA\op \times \cA \lra \Ab} and an obstruction operator $\OO$.
The following properties are required to hold:
\begin{enumerate}
\renewcommand{\labelenumi}{(\arabic{enumi})~}
\item $\cT$ is enriched in left $n$-cubical sets and has zero maps, that is,
for all objects $X$, $Y$ in $\cT$, we have the $n$-cubical set
\w{\Mor\sb{\cT}(X, Y)} with operators
\w{(d\sb{o}\sp{i})\sp{\ast} = \partial\sp{i}} and
zero elements \w{o\sp{t} \in \Mor\sb{\cT}(X, Y)\sb{t}} such that
\begin{equation*}
\begin{split}
\partial\sp{i}(f g)~=~(\partial\sp{i} f) g & \hsm\text{for} \hsm i \leq \dim(f)\\
\partial\sp{i}(f g)~=~f (\partial\sp{i - \dim(f)} g) &
\hsm \text{for}\hsm i > \dim(f)\\
o\sp{t} g~=~o\sp{t + \dim(g)} & \\
f o\sp{h}~=~o\sp{\dim(f) + h}~. &
\end{split}
\end{equation*}
\noindent Here \w{f g} is the composite in the $n$-graded category $\cT$,
which is defined if \w[.]{\dim(f) + \dim(g) \leq n}
\item The $0$-skeleton $\cT^0$ is the subcategory of $\cT$ consisting
of morphisms $f$ with \w[,]{\dim(f) = 0} this is a category together with a
functor \w{q:\cT^0 \lra \cA} which is the identity on objects and full
(a quotient functor). Moreover, $D$ is a bifunctor
$$
\xymatrix@1{D:\cA\op \times \cA\ar[r] & \Ab}
$$
\noindent into the category of abelian groups. Here $D$ defines via $q$
a bifunctor on \w{\cT\sp{0}} which satisfies \w{(o^0)\sp{\ast} = o} and
\w[.]{(o^0)\sb{\ast} = o}
For a zero map \w{o^0:X \lra Y} in \w{\cT\sp{0}} we obtain the zero map
\w{o\sb{X, Y} = q(o^0)} in $\cA$.

For \w{f:X \lra Y}in \w[,]{\cT\sp{0}} we have \w{q(f) = o\sb{X, Y}} if and only
if there is \w{F:X \lra Y} in $\cT$ with \w{\dim(F) = 1} and
\w[.]{\partial\sp{1} F = f} This is the boundary property in dimension~1.
\item The obstruction operator $\OO$ yields for each left cubical ball $B$
an element
$$
\OO\sb{B}(a\sb{1}, \dotsc, a\sb{k}) \in D(X, Y)
$$
\noindent where \w{a\sb{1}, \dotsc, a\sb{k} \in \Mor_\cT(X, Y)\sb{n}} is a
$k$-tuple satisfying the gluing condition in $B$, see~\wref[.]{egc}

This obstruction operator satisfies the complement rule, the double rule,
and the action formula as in Section \ref{cob}. Here the \emph{action} $+$ of
\w{D(X, Y)} on the set \w{\Mor_\cT(X, Y)\sb{n}} is defined by:
$$
\text{if}\hsm \OO\sb{B}(b, a) = \alpha,\hsm\text{then}\hsm
b = a + \alpha~.
$$
\noindent This $b$ is required to be unique. Here $B$ is the double of
\w{I\sp{n}} with \w[.]{\vare\sb{1} = +1}

The action $+$ is transitive and effective on the set of all elements $a$ in
\w{\Mor_\cT(X, Y)\sb{n}} which coincide on the boundary (that is,
\w[,]{\partial\sp{i} a = b\sb{i}} where \w{(b\sb{1}, \dotsc, b\sb{n})} is fixed).
\item The obstruction operator satisfies for
$$
f \in \Mor_\cT(X', X)\sb{0}\hsm\text{and}\hsm g \in \Mor_\cT(Y, Y')\sb{0}
$$
\noindent the \emph{naturality rule}
\begin{myeq}\label{eqnatur}
\begin{split}
\OO\sb{B}(g a\sb{1}, \dotsc, g a\sb{k}) &= 
g\sb{\ast} \OO\sb{B}(a\sb{1}, \dotsc, a\sb{k})\\
\OO\sb{B}(a\sb{1} f, \dotsc, a\sb{k} f) &= 
f\sp{\ast} \OO\sb{B}(a\sb{1}, \dotsc, a\sb{k})~.
\end{split}
\end{myeq}
\noindent Here \w{f\sp{\ast}} and \w{g\sb{\ast}} denote the induced maps on $D$.
This implies that \w{g(a + \alpha) = g a + g\sb{\ast}\alpha} and
\w[.]{(a + \alpha) f = a f + f\sp{\ast} \alpha}
\item The obstruction operator satisfies the following \emph{triviality rule}:
For morphisms
$$
\xymatrix@1{Z & Y\ar[l]_f & X\ar[l]_g}
$$
\noindent in $\cT$ with $\dim(f), \dim(g) \leq n$ and
$$
\dim(f) + \dim(g) = n + 1
$$
\noindent we have the \wwb{n+1}tuple \w{(a\sb{1}, \dotsc, a\sb{n + 1})} in
\w{\Mor_\cT(X, Z)\sb{n}} given by
$$
a\sb{t} =
\begin{cases}
(\partial^t f) g & \text{for}\hsm 1 \leq t \leq \dim(f),\\
f(\partial\sp{t - \dim(f)} g) & \text{for}\hsm \dim(f) < t \leq n + 1.
\end{cases}
$$
\noindent This \wwb{n+1}tuple satisfies the gluing condition in
\w[.]{B = \cT\sb{0}\sp{n}} The associated obstruction
$$
\OO\sb{B}(a\sb{1}, \dotsc,a\sb{n + 1}) = 0
$$
\noindent is trivial.
\end{enumerate}
\end{defn}

\begin{remark}
The case \w{n=1} is classical, and is considered in the next section\vsm .
\end{remark}

We now generalize the concept of an $n$-th order chain complex (defined
in Section \ref{chocc} for a mapping algebra $\cC$), replacing \w{\NnC} by
an abstract algebra of left $n$-cubical balls $\cT$:

\begin{defn}\label{defi:preChainComplex}
Let \w{(\cT, \cA, D, \OO)} be an algebra of left $n$-cubical balls.
A functor of $n$-graded categories
$$
K:\bZ(N, M)\sb{\otimes}\sp{n}~\to~\cT
$$
\noindent satisfying the inclusion property \wref{epcc} is an
\emph{$n$-th order pre-chain complex} in $\cT$. This is an
\emph{$n$-th order chain complex} in $\cT$ if for
\w[,]{i,i - n - 2 \in \bZ(N, M)} the obstructions
$$
\OO K(i, \partial I\sb{n + 1}) = \OO\sb{B}(b\sb{1}, \dotsc, b\sb{n + 1}) = 0
$$
\noindent vanish. Here $B$ is the left cubical ball \w[,]{B = \cT\sb{0}\sp{n}} and
$$
K(i, \partial I\sb{n + 1}) = \begin{cases}
b\sb{1} = K(i, \emptyset \otimes I\sb{n})\\
b\sb{r + 1} = K(i, I\sb{r} \otimes I\sb{n - r}) & 
\text{for $1 \leq r \leq n - 1$}\\
b\sb{n + 1} = K(i, I\sb{n} \otimes \emptyset)
\end{cases}
$$
\noindent (see~\wref[).]{eqoe} Since $K$ is a functor we have
$$
\begin{array}{lclcl}
K(i, \emptyset \otimes I\sb{n}) & = & K(i - n - 1, \emptyset) K(i, I\sb{n}) & = &
\delta\sb{i - n - 1} K(i, I\sb{n})\\
K(i, I\sb{r} \otimes I\sb{s}) & = & K(i - s - 1, I\sb{r}) K(i, I\sb{s})\\
K(i, I\sb{n} \otimes \emptyset) & = & K(i - 1, I\sb{n}) K(i, \emptyset) & = &
K(i - 1, I\sb{n}) \delta\sb{i}
\end{array}
$$
\noindent where the right hand side denotes composition in $\cT$. We define
higher order Toda brackets in $\cT$
$$
\langle\delta\sb{1}, \dotsc, \delta\sb{n + 2}\rangle \subset 
D(K\sb{n + 2}, K\sb{0})~
$$
\noindent as in Definition \ref{dtoda}.
\end{defn}

%
%
\sect{Track categories and algebras of left $1$-cubical balls}
\label{stcaone}

In this section we discuss algebras of left $n$-cubical balls in the special
case \w[.]{n=1} We show that every abelian track category (cf.\ \cite{BW})
has the structure of an algebra of left $1$-cubical balls. This shows that
algebras of left $n$-cubical balls are $n$-dimensional analogues of
track categories for every \w[.]{n \geq 1}

\begin{defn}\label{dtrack}
A \emph{track category} is a category $\cC$ enriched in groupoids. For
objects $X$, $Y$ in $\cC$ we have the groupoid \w{\Mor\sb{\cC}(X, Y)} with
objects $f$, $g$ and morphisms \w[.]{F:f \lra g}

The morphisms \w{F:f \lra f} form the automorphism group \w[,]{\Aut\sb{\cC}(f)}
and we write \w{f \simeq g} if there is \w[.]{F:f \lra g} Let
\w[,]{\dim(f) = 0} \w[,]{\dim(F) = 1} \w[,]{(d\sb{0}\sp{1})\sp{\ast} F = f} and
\w[.]{(d\sb{1}\sp{1})\sp{\ast} F = g} Morphisms of dimension $0$ form the category
\w[,]{\cC\sb{0}} and the homotopy relation $\simeq$ defines the homotopy category
$$
\cA = \pi\sb{0} \cC = \cC\sb{0}/\!\simeq.
$$
\end{defn}

The track category $\cC$ is called \emph{abelian} if all automorphism groups
\w{\Aut\sb{\cC}(f)} are abelian groups. We assume that $\cC$ has zero maps
\w{o\sb{X, Y} \in \Mor\sb{\cC}(X, Y)\sb{0}} for each \w[.]{X,Y\in\cC}
We then obtain a bifunctor \w{D:\cA\op \times \cA\to\Ab} defined
$$
D(X, Y)~:=~\Aut\sb{\cC}(o\sb{X, Y})~.
$$

We define the $1$-graded category $\cT$ associated to $\cC$
(cf.\ \S \ref{dgradec}) by
$$
\Mor_\cT(X, Y)\sb{i}~:=~\begin{cases}
\Mor\sb{\cC}(X, Y)\sb{0} & \text{if}~~i=0\\
\{(F, f), F:f\to o\sb{X, Y}\} \subset \Mor\sb{\cC}(X, Y)\sb{1} & \text{if}~~ i=1.
\end{cases}
$$
\noindent Let \w{\partial\sp{1}} be defined by \w[,]{\partial\sp{1}(F, f) = f}
and let the zero elements be given by \w[,]{o\sp{0} = o\sb{X,Y}}
\w{o\sp{1} =} identity of \w[.]{o\sb{X, Y}}

\begin{prop}
For any abelian track category $\cC$ with zero maps,
\w{\Nul{1}(\cC) = (\cT, \cA, D, \OO)} with $\cT$, $\cA$ and $D$ as
above and with the obstruction operator $\OO$ as follows is an algebra of left
$1$-cubical balls. \hfill $\Box$
\end{prop}

Up to equivalence there is only one left cubical ball $B$ of dimension $1$:
this is the double of $I$, which is equivalent to \w[.]{T\sb{0}\sp{1}} Given
\w{a\sb{1} = (F, f)} and \w{a\sb{2} = (G, g)} with gluing condition
\w[,]{\partial\sp{1} a\sb{1} = f = g = \partial\sp{1} a\sb{2}} let
$$
\OO\sb{B}(a\sb{1}, a\sb{2})~:=~ F G\sp{-1} \in \Aut\sb{\cC}(o\sb{X, Y})
$$
\noindent be the \emph{obstruction}. The action for
\w{\alpha \in \Aut\sb{\cC}(o\sb{X, Y})} and \w{a = (F, f)} is given by
\w[,]{a + \alpha = (\alpha F, f)} with
\w[.]{\OO\sb{B}(a + \alpha, a) = \alpha F F\sp{-1} = \alpha}
The triviality rule of $\OO$ is satisfied, since for a diagram
$$
\xymatrix{Z && Y\ar[ll]^g\ar@/_1.3em/[ll]\sp{\Uparrow G}\sb{o} && 
X~,\ar[ll]^f\ar@/_1.3em/[ll]\sp{\Uparrow F}\sb{o}}
$$
\noindent in $\cC$ we have the formula \w{G f = g F} by the usual
interchange law (since \w[),]{o\sb{Y,Z}F=G o\sb{X,Y}=\id\sb{o\sb{X,Z}}}
so that
$$
\OO\sb{B}(a\sb{1}, a\sb{2}) = o\hs \text{for}\hsm
a\sb{1} = G(\partial\sp{1}F)\hsm
\text{and}\hsm a\sb{2} = (\partial\sp{1} G)F ~.
$$

\begin{example}
Let $\cC$ be a category enriched in groupoids with zero maps and let $\cC$
be abelian. Then the algebra of left $1$-cubical balls \w{\Nul{1}(\cC)} is
defined and a \emph{triple Toda bracket}
$$
\langle\delta\sb{1}, \delta\sb{2}, \delta\sb{3}\rangle\hsm\text{in}\hsm 
\Nul{1}(\cC)
$$
\noindent coincides with the classical triple Toda bracket in $\cC$.
Moreover, a $1$-st order chain complex in \w{\Nul{1}(\cC)} as defined in
\S \ref{defi:preChainComplex} coincides with a \emph{secondary chain complex}
in $\cC$ as studied in~\cite{BJSe}.
\end{example}

\begin{remark}
Abelian track categories are classified by cohomology, see
\cite{BW,BD,PiraH,BAl,BJCl}. It would be interesting to classify accordingly
algebras of left $n$-cubical balls for \w[.]{n \geq 1}
\end{remark}

%
%
\sect{The inductive step of the resolution theorem}
\label{cisrt}

Resolution Theorem \ref{tresol} refers to higher order resolutions in a mapping
algebra. In the following section we shall prove a more general resolution
theorem, for certain types of algebras of left $n$-cubical balls, which
we now define. The proof is by induction, and we give the induction step here.

\begin{defn}
An algebra of left $n$-cubical balls \w{\cT = (\cT, \cA, D, \OO)} is
\emph{a $\Sigma$-algebra of left $n$-cubical balls} if
\w{\cA = (\cA, \aC, \Sigma)} is a $\Sigma$-algebra (cf.\ \S \ref{dsa})
and \[D(X, Y) = \Hom\sb{\cA}(\Sigma\sp{n} X, Y)\]
for $X$ in $\aC$ and $Y$ in $\cA$. We say that $\cT$ is an
\emph{$\Omega$-algebra of left $n$-cubical balls} if
\w{\cA = (\cA, \aC, \Omega)} is an $\Omega$-algebra (cf.\ \S \ref{dsa})
and
$$
D(X, Y)~=~\Hom\sb{\cA}(X, \Omega\sp{n} Y)
$$
\noindent for $Y$ in $\aC$ and $X$ in $\cA$.
\end{defn}

\begin{thm}\label{thota}
Let $\cT$ be a $\Sigma$-track algebra in dimension $n$ and consider a
functor of $n$-graded categories
$$
K:\bZ(\infty, -1)_\otimes\sp{n}~\to~\cT
$$
\noindent which is a pre-chain complex (cf.\ \S \ref{dprech}), and which is
based on an $\aC$-resolution \w{A\sbu} of $X$ in $\cA$
(cf.\ \S \ref{dbasedon}). Then there exists a functor
$$
\xymatrix@1{K':\bZ(\infty, -1)_\otimes\sp{n}\ar[r] & \cT}
$$
\noindent which coincides with $K$ in dimension \w[,]{\leq n - 1} and which is an
$n$-th order chain complex in $\cT$ (based on \w[).]{A\sbu}
\end{thm}

The dual also holds.

\begin{thm}\label{thdota}
Let $\cT$ be an $\Omega$-track algebra in dimension~$n$ and consider a
functor of $n$-graded categories
$$
\xymatrix@1{L:\bZ(+1, -\infty)_\otimes\sp{n}\ar[r] & \cT}
$$
\noindent which is a pre-chain complex and which is based on an
$\aC$-coresolution \w{A\ubu} of $Y$ in $\cA$. Then there exists a
functor
$$
\xymatrix@1{L':\bZ(+1, -\infty)_\otimes\sp{n}\ar[r] & \cT}
$$
\noindent which coincides with $L$ in dimension $\leq n - 1$ and which is an
$n$-th order chain complex in $\cT$ (and is based on \w[).]{A\ubu}
\end{thm}

\begin{proof}[Proof of Theorem \protect\ref{thota}]
The functor \w{K'} is determined by $K$ in dimension \w{\leq n - 1} and by
\begin{myeq}\label{eqpisrto}
K'(i, I\sb{n}) = K(i, I\sb{n}) + \alpha\sb{i}
%
\end{myeq}
\noindent in dimension $n$. See Lemma \ref{lgcc}. Here the elements 
\w{\alpha\sb{i}} are obtained inductively as follows. We have to 
choose \w{\alpha\sb{i}} for \w{i \geq n} in such a way that the obstruction
\begin{myeq}\label{eqpisrtt}
\xi(\alpha\sb{i - 1}, \alpha\sb{i}) = \OO\sb{B}(\delta\sb{i - n - 1}
K'(i, I\sb{n}), b\sb{2}, \dotsc, b\sb{n}, K'(i - 1, I\sb{n}) \delta\sb{i})
\end{myeq}
\noindent vanishes with \w{b\sb{r + 1} = K(i, I\sb{r} \otimes I\sb{n - r})} for
\w[,]{1 \leq r \leq n - 1} see~\ref{defi:preChainComplex}. We start with
\w{i=n+1} and \w[.]{\alpha\sb{n}:=0}
In this case \w{(\delta\sb{0})\sb{\ast}} is surjective since \w{A\sbu} 
is a resolution with \w[.]{\delta\sb{0}:A\sb{0} \lra A\sb{-1}$, $A\sb{-1} = X}
The action rule implies that
\begin{myeq}
\label{eqpisrtth}
\xi(\alpha\sb{i - 1}, \alpha\sb{i}) = \xi(0, 0) +
\vare\sb{1}(\delta\sb{i - n - 1})\sb{\ast} \alpha\sb{i} +
\vare\sb{n + 1} (\delta\sb{i})\sp{\ast} \alpha\sb{i - 1}~.
\end{myeq}
\noindent Here \w{\vare\sb{1}, \dotsc, \vare\sb{n + 1}} are the orientation
signs for the left cubical ball \w[.]{B = T\sb{0}\sp{n}} For \w[,]{i = n + 1} 
we obtain for \w{\alpha\sb{n} = 0} the equation:
$$
\xi(0, \alpha\sb{n + 1}) = \xi(0, 0) + \vare\sb{1}(\delta\sb{0})\sb{\ast} 
\alpha\sb{n + 1}~.
$$
\noindent Since \w{(\delta\sb{0})\sb{\ast}} is surjective, there is an
\w{\alpha\sb{n + 1}} with \w[.]{\xi(0, \alpha\sb{n + 1}) = 0}

We now consider \wref{eqpisrtth} for \w[.]{i = n + 2}
Using Lemma \ref{lhaupt}, we shall show that
\begin{myeq}\label{eqpisrtf}
(\delta\sb{0})\sb{\ast} \xi(\alpha\sb{n + 1}, \alpha\sb{n + 2}) = 0.
\end{myeq}
Since \w{A\sbu} is a resolution, this implies that
\begin{myeq}
\xi(\alpha\sb{n + 1}, \alpha\sb{n + 2}) \in \image(\delta\sb{1})\sb{*}~.
\end{myeq}

Since by \wref{eqpisrtth} we have
\begin{myeq}
\xi(\alpha\sb{n + 1}, \alpha\sb{n + 2}) = \xi(0, 0) +
\vare\sb{1}(\delta\sb{1})\sb{\ast} \alpha\sb{n + 2} +
\vare\sb{n + 1} \delta\sb{n + 2}\sp{\ast} \alpha\sb{n + 1}~,
\end{myeq}
\noindent we can choose \w{\alpha\sb{n + 2}} with
\w[.]{\xi(\alpha\sb{n + 1}, \alpha\sb{n + 2}) = 0}
Thus we get by induction \w{\alpha\sb{i}} \wb[,]{i \geq n} such that
\w[.]{\xi(\alpha\sb{i - 1}, \alpha\sb{i}) = 0} Hence \w[,]{K'} defined by
\wref[,]{eqpisrto} satisfies the obstruction property, and hence is an
$n$-th order chain complex as in the Theorem. In the next Lemma we show
that \wref{eqpisrtf} holds.
\end{proof}

We introduce the following notation on the `boundary' of \w{I\sb{n + 1}}
\wb[.]{n \geq 0}
Let
$$
\partial I\sb{1}~:=~\emptyset \otimes \emptyset~,
$$
\noindent and for \w[,]{n \geq 1} let
\begin{myeq}\label{eqbdcube}
\partial I\sb{n + 1} = (\emptyset \otimes I\sb{n}, I\sb{1} \otimes 
I\sb{n - 1}, I\sb{2} \otimes I\sb{n - 2}, \dotsc, I\sb{n - 1} 
\otimes I\sb{1}, I\sb{n} \otimes \emptyset)~.
\end{myeq}
\noindent (see~\wref[).]{eqoe} 

Given a pre-chain complex \w[,]{K':\bZ(\infty, -1)_\otimes\sp{n} \lra \cT}
we obtain for \w{i \geq n\geq 1} the obstruction element
$$
\OO\sb{B} K'(i, \partial I\sb{n + 1})
$$
\noindent where \w[.]{B = T\sb{0}\sp{n}} This corresponds to \wref{eqpisrtt} in
the proof above.

\begin{hlemma}\label{lhaupt}
Let \w{n \geq 1} and \w[,]{i \geq n + 2} and assume that
$$
\OO\sb{B} K'(i - 1, \partial I\sb{n + 1}) = 0~.
$$
\noindent Then we also have
$$
\OO\sb{B} K'(i, \emptyset \otimes \partial I\sb{n + 1}) = 0~.
$$
\end{hlemma}

\begin{remark}\label{rhaupt}
For the proof of Hauptlemma \ref{lhaupt}, we use the following equation given 
by the triviality rule of \S \ref{dalgncb}(5):

\begin{myeq}\label{ehlo}
\OO\sb{B} K'(i, \partial\sb{r, s})~=~0
\end{myeq}
\noindent for  with \w[,]{r + s = n + 1} \w[,]{r,s \geq 1} and 
\w[,]{i \geq n} where
$$
\partial\sb{r, s}~:=~((\partial I\sb{r}) \otimes I\sb{s},\ 
I\sb{r} \otimes(\partial I\sb{s})).
$$
\noindent The hypothesis of the Lemma  implies that
\begin{myeq}\label{ehlt}
\OO\sb{B} K'(i, (\partial I\sb{n + 1}) \otimes \emptyset) = 0
\end{myeq}
by the naturality rule \wref[.]{eqnatur}

One should think of the Lemma as relating the obstructions associated to 
the boundaries of the various sub-$(n+1)$-balls of the `corner' left cubical
ball \w{T\sb{0}\sp{n+1}\subseteq\partial I\sp{n+2}} (see \S \ref{egtball}).
The two obstructions of the Lemma are those associated to the \wwb{n+1}balls 
\w{\emptyset\otimes I\sb{n+1}} and \w[,]{I\sb{n+1}\otimes\emptyset}
respectively.  However, we must describe them in terms of the lists of \w{n+1} 
left $n$-cubes constituting their boundaries, namely,
\begin{myeq}\label{eqx}
X~:=~(\emptyset\otimes\emptyset\otimes I\sb{n},\,
\emptyset\otimes\partial\sb{1,n},\,\emptyset\otimes\partial\sb{2,n-1},\,\dotsc
\emptyset\otimes\partial\sb{n,1},\,\emptyset\otimes I\sb{n}\otimes\emptyset)
\end{myeq}
\noindent (cf.\ \wref[)]{eqbdcube} and
\begin{myeq}\label{eqy}
Y~:=~(\emptyset\otimes I\sb{n}\otimes\emptyset,\,
\partial\sb{1,n}\otimes\emptyset,\,\partial\sb{2,n-1}\otimes\emptyset,\,\dotsc
\partial\sb{n,1}\otimes\emptyset,\,I\sb{n}\otimes\emptyset\otimes\emptyset)~,
\end{myeq}
\noindent respectively. The vanishing of the obstruction \wref{ehlt} means that 
\w[,]{Y\sim 0} and similarly 
\w{\OO\sb{B} K'(i, \emptyset \otimes \partial I\sb{n + 1})} vanishes if and only 
if \w[.]{X\sim 0}

All the other \wwb{n+1}cubes of \w{T\sb{0}\sp{n+1}}
are non-trivial $\otimes$-products, with boundaries of the form 
\w{\partial\sb{r,s}} for \w{r + s = n + 1} \wb[,]{r,s \geq 1} 
so the corresponding obstructions vanish by the triviality rule. This allows us 
to extend the \wwb{n+1}balls bounded by $X$ and $Y$, respectively, to larger 
\wwb{n+1}balls bounded by \w{X'} and \w[,]{Y'} in such a way that 
\w{X'\sim X} and \w{Y'\sim Y} (using the complement rule of Lemma \ref{lcompl}).

The Lemma is then proved by showing that one can choose these extensions
so that \w{X'=Y'} (because we have two complementary sub-$(n+1)$-balls of 
\w{T\sb{0}\sp{n+1}} with the same boundary). This implies that the
corresponding obstructions are equal to each other, and thus both vanish by
\wref[.]{ehlt}
\end{remark}

\begin{proof}[Proof of Hauptlemma \protect\ref{lhaupt} for \ $n = 1$]
In this case the triviality rule \wref{ehlo} yields:
\begin{myeq}\label{epohlo}
\partial\sb{1,1}~:=~(\emptyset\otimes\emptyset\otimes I\sb{1},\,
I\sb{1}\otimes\emptyset\otimes \emptyset)~\sim~0~,
\end{myeq}
\noindent while by \wref{ehlt} we have
\begin{myeq}\label{epohlt}
Y~:=~(\emptyset\otimes I\sb{1}\otimes\emptyset,\,
I\sb{1}\otimes\emptyset\otimes\emptyset)~\sim~0.
\end{myeq}
\noindent Thus by the complement rule and \wref{epohlo} we see that
$$
X~:=~(\emptyset\otimes\emptyset \otimes I\sb{1},\, 
\emptyset\otimes I\sb{1} \otimes \emptyset)
$$
\noindent and
$$
X'~:=~(I\sb{1} \otimes \emptyset \otimes \emptyset,\,
\emptyset\otimes I\sb{1}\otimes \emptyset)
$$
\noindent satisfy \w{X\sim X'=Y\sim 0} by~\wref[.]{epohlt} 
\end{proof}

\begin{proof}[Proof of Hauptlemma \protect\ref{lhaupt} for \ $n = 2$]
By \wref[,]{ehlo} we find:
\begin{myeq}\label{epthlo}
(\emptyset\otimes I\sb{1}\otimes I\sb{1},\, 
I\sb{1}\otimes\emptyset\otimes I\sb{1},\, 
I\sb{2}\otimes\emptyset\otimes\emptyset)~\sim~0
\end{myeq}
\noindent and
\begin{myeq}\label{epthlt}
(\emptyset\otimes\emptyset\otimes I\sb{2},\,
I\sb{1}\otimes\emptyset\otimes I\sb{1},\, 
I\sb{1}\otimes I\sb{1}\otimes\emptyset)~\sim~0.
\end{myeq}
\noindent By the assumption~\wref{ehlt} we have
\begin{myeq}
Y~=~(\emptyset\otimes I\sb{2}\otimes\emptyset,\,
I\sb{1}\otimes I\sb{1}\otimes\emptyset,\, 
I\sb{2}\otimes \emptyset\otimes\emptyset)~\sim 0~.
\end{myeq}
\noindent We have to show
\begin{myeq}
X ~=~(\emptyset\otimes\emptyset\otimes I\sb{2},\, 
\emptyset\otimes I\sb{1}\otimes I\sb{1},\,
\emptyset\otimes I\sb{2}\otimes\emptyset)~\sim~0~.
\end{myeq}
\noindent By the complement rule and~\wref{epthlt} (removing
\w[)]{\emptyset\otimes\emptyset\otimes I\sb{2}} we find that
$$
X~\sim~X'~:=~(I\sb{1}\otimes\emptyset\otimes I\sb{1},\,
I\sb{1}\otimes I\sb{1}\otimes\emptyset,\,\emptyset\otimes I\sb{1}\otimes I\sb{1},\,
\emptyset\otimes I\sb{2}\otimes\emptyset)~,
$$
\noindent and similarly by \wref{epthlo} (removing
\w[):]{I\sb{2}\otimes\emptyset\otimes\emptyset} 
$$
Y~\sim~Y'~:=~(\emptyset\otimes I\sb{2}\otimes\emptyset,\, 
I\sb{1}\otimes I\sb{1}\otimes\emptyset,\,
\emptyset\otimes I\sb{1}\otimes I\sb{1},\, I\sb{1}\otimes\emptyset\otimes I\sb{1}) 
~.
$$
\noindent Thus \w[.]{X \sim X' = Y' \sim Y \sim 0}
\end{proof}

\begin{proof}[Proof of Hauptlemma \protect\ref{lhaupt}]
By~\wref{ehlo} we have the relations
\begin{myeq}
\partial\sb{r, s} \sim 0 \hsp \text{for}\hsm r + s = n + 1, \ r \geq 1~.
\end{myeq}
\noindent By~\wref[,]{ehlt} the hypothesis of the Lemma implies that
\begin{myeq}\label{ephlt}
Y~:=~(\partial I\sb{n + 1}) \otimes \emptyset~\sim ~0~.
\end{myeq}
\noindent We have to show that
\begin{myeq}
X~=~\emptyset \otimes (\partial I\sb{n + 1}) \sim 0.
\end{myeq}

We now define inductively a sequence of collections \w{Y\sp{(s)}} of $n$-facets 
of \w{T\sb{0}\sp{n+1}} for \w[.]{0\leq s\leq n}
We start with \w{Y\sp{(0)}:=Y} as in \wref[,]{eqy} and define \w{Y\sp{(1)}} 
by adding
$$
\partial\sb{n,1}~=~(\emptyset\otimes I\sb{n-1}\otimes I\sb{1},\,
I\sb{1}\otimes I\sb{n-2}\otimes\otimes I\sb{1},\,\dotsc,
I\sb{n-1}\otimes I\sb{1}\otimes I\sb{1},\,I\sb{n}\otimes\emptyset\otimes\emptyset)
$$
\noindent and removing the common $n$-cube 
\w{I\sb{n}\otimes\emptyset\otimes\emptyset} to obtain:
$$
Y\sp{(1)}~=~(\underbrace{\emptyset\otimes I\sb{n}\otimes\emptyset,\,\dotsc,
I\sb{n-1}\otimes I\sb{1}\otimes\emptyset},\,
\underbrace{\emptyset\otimes I\sb{n-1}\otimes I\sb{1},\,\dotsc,
I\sb{n-1}\otimes I\sb{1}\otimes I\sb{1}})~.
$$
\noindent Note that \w[,]{Y\sim Y\sp{(1)}} since the associated obstructions 
are equal by the triviality and complement rules. 

By induction we define \w{Y\sp{(s)}} by adding 
\w{\partial\sb{n+1-s, s}} to \w[,]{Y\sp{(s-1)}} so that
\begin{myeq}\label{eqys}
\begin{split}
Y&\sp{(s)}~=~(\underbrace{\emptyset\otimes I\sb{n}\otimes\emptyset,\,\dotsc,
I\sb{n-s}\otimes I\sb{s}\otimes\emptyset},\
\underbrace{\emptyset\otimes I\sb{n-1}\otimes I\sb{1},\,\dotsc,
I\sb{n-s}\otimes I\sb{s-1}\otimes I\sb{1}},\\
&\underbrace{\emptyset\otimes I\sb{n-2}\otimes I\sb{2},\,\dotsc,
I\sb{n-s}\otimes I\sb{s-2}\otimes I\sb{2}},\dotsc,
\underbrace{\emptyset\otimes I\sb{n-s}\otimes I\sb{s},\,\dotsc,
I\sb{n-s}\otimes\emptyset\otimes I\sb{s}})
\end{split}
\end{myeq}
\noindent (consisting of \w{s+1} groups of \w{n-s+1} $n$-cubes each). 
Again by the complement rule \w[,]{Y\sp{(s)}\sim Y\sp{(s-1)}} so by
induction \w[.]{Y\sp{(s)}\sim Y\sim 0} Moreover, since the $n$-cube facets 
of \w{T\sb{0}\sp{n+1}} common to \w{Y\sp{(s-1)}} and 
$$
\partial\sb{s,n+1-s}~=~
((\partial I\sb{s}) \otimes I\sb{n+1-s},\ I\sb{s}\otimes(\partial I\sb{n+1-s}))
$$
\noindent form the left cubical $n$-ball 
\w[,]{I\sb{s}\otimes(\partial I\sb{n+1-s})} by induction we see also that
\w{Y\sp{(s)}} is the boundary of an $n$-ball (the connected sum of 
\w{\partial\sb{s,n+1-s}} and \w[).]{Y\sp{(s-1)}}

Similarly, if we start with \w{X\sp{(0)}:=X} as in \wref[,]{eqx} and define 
\w{X\sp{(1)}} by adding
$$
\partial\sb{1,n}~=~(\emptyset\otimes\emptyset\otimes I\sb{n},\
\underbrace{I\sb{1}\otimes\emptyset\otimes I\sb{n-1},\,\dotsc,
I\sb{1}\otimes I\sb{n-2}\otimes I\sb{1},\,I\sb{1}I\sb{n-1}\otimes\emptyset})
$$
\noindent and removing the common $n$-ball 
\w[,]{\emptyset\otimes\emptyset\otimes I\sb{n}} we obtain
$$
X\sp{(1)}~=~(\underbrace{\emptyset\otimes I\sb{1}\otimes I\sb{n},\,\dotsc,
\emptyset\otimes I\sb{n}\otimes\emptyset},\ 
\underbrace{I\sb{1}\otimes \emptyset\otimes I\sb{n-1},\,\dotsc,
I\sb{1}\otimes I\sb{n-1}\otimes\emptyset})~.
$$
\noindent By induction define \w{X\sp{(r)}} by adding 
\w{\partial\sb{r,n+1-r}} to \w[,]{X\sp{(r-1)}} so that
\begin{myeq}\label{eqxr}
\begin{split}
X&\sp{(r)}~=~(\underbrace{\emptyset\otimes I\sb{r}\otimes I\sb{n-r},\,\dotsc,
\emptyset\otimes I\sb{n}\otimes\emptyset},\
\underbrace{I\sb{1}\otimes I\sb{r-1}\otimes I\sb{n-r},\,\dotsc,
I\sb{1}\otimes I\sb{n-1}\otimes\emptyset},\\
&\underbrace{I\sb{2}\otimes I\sb{r-2}\otimes I\sb{n-r},\,\dotsc,
I\sb{2}\otimes I\sb{n-2}\otimes\emptyset},\dotsc, 
\underbrace{I\sb{r}\otimes\emptyset\otimes I\sb{n-r},\,\dotsc,
I\sb{r}\otimes I\sb{n-r}\otimes\emptyset})
\end{split}
\end{myeq}
\noindent Again \w{X\sp{(r)}} is the boundary of an $n$-ball, and
by the complement rule \w[,]{X\sp{(r)}\sim X\sp{(r-1)}} so by
induction \w[.]{X\sp{(r)}\sim X} 

Finally, assume \w{r+s=n} (e.g., for \w{s=\lfloor \frac{n}{2}\rfloor} and
\w[).]{r=\lceil \frac{n}{2}\rceil} 
In this case, by direct inspection we see that the \w{(r+1)\cdot(n-r+1)}
$n$-cubes of \wref{eqxr} are just a reordering of the \w{(s+1)\cdot(n-s+1)} 
$n$-cubes of \wref[,]{eqys} so \w[.]{X\sp{(r)}=Y\sp{(s)}} 
The reason this happens is that the \wwb{n+1}cube facets of \w{T\sb{0}\sp{n+1}} 
forming the interiors of the two \wwb{n+1}balls bounded by \w{X\sp{(r)}} and 
\w[,]{Y\sp{(s)}} respectively, constitute a partition of all the \wwb{n+1}cubes,
so they partition the \wwb{n+1}ball \w{T\sb{0}\sp{n+1}} into two complementary
\wwb{n+1}balls with a common boundary \w[.]{X\sp{(r)}=Y\sp{(s)}}

Thus 
$$
X~\sim~X\sp{(r)}~=~Y\sp{(s)}~\sim~Y~\sim~0~,
$$
\noindent so the obstruction 
\w{\OO\sb{B} K'(i, \emptyset \otimes \partial I\sb{n + 1})} associated to $X$ 
vanishes, as claimed.
\end{proof}

%
%
\sect{Algebras of left cubical balls and the proof of the Resolution Theorem}
\label{ctaprt}

In order to prove Resolution Theorem \ref{tresol}, we need to relate
algebras of left cubical $n$-balls for different $n$, as follows:

\begin{defn}
A \emph{total algebra of left $n$-cubical balls} \w{\Tln} is a sequence of
algebras of left $m$-cubical balls
$$
\cT(m) = (\cT(m), \cA, D\sb{m}, \OO\sp{m})
$$
for \w[,]{m=1,2, \dotsc, n} together with quotient functor
$$
q:\cT(m + 1)\sp{m}~\to~ \cT(m)
$$
which is the identity on objects, is full, and is the identity functor on 
\wwb{m - 1}skeleta:
$$
q(\cT(m + 1)\sp{m - 1})~=~\cT(m)\sp{m - 1}~.
$$
\noindent Moreover, the \emph{boundary property} of 
Lemma \wref{lem:boundaryProperty} holds \wh that is, for \w[,]{B = T\sb{0}\sp{m}} 
we have
$$
\OO\sb{B}\sp{m}(a\sb{1}, \dotsc, a\sb{m + 1}) = 0
$$
\noindent if and only if there exists \w{\overline a \in \cT(m + 1)\sb{m + 1}} 
with \w{q(\partial\sp{i} \overline a)} representing \w{a\sb{i}} for
\w[.]{i = 1, \dotsc, m + 1}
\end{defn}

\begin{example}
Let $\cC$ be a category enriched in pointed spaces with zero maps. Then
$$
\NlnC~:=~(\NnC, \Nul{n - 1} \cC, \dotsc, \Nul{1}\cC)
$$
\noindent is a total algebra of left $n$-cubical balls.
\end{example}

\begin{defn}
We say that \w{\Tln} is a \emph{$\Sigma$-algebra of left cubical balls}
if \w{\cA = (\cA, \aC, \Sigma)} is a $\Sigma$-algebra as in 
Definition \ref{dsa} and
$$
D\sb{m}(X, Y) = \Hom\sb{\cA}(\Sigma\sp{m} X, Y)
$$
\noindent for \w{m = 1, \dotsc, n} and $X$ in $\aC$ and $Y$ in $\cA$.

Dually we say \w{\Tln} is an \emph{$\Omega$-algebra of left cubical balls} if
\w{\cA = (\cA, \aC, \Omega)} is an $\Omega$-algebra as in~\wref{eqbeg} and
$$
D\sb{m}(X, Y) = \Hom\sb{\cA}(X, \Omega\sp{m} Y)
$$
\noindent for \w{m = 1, \dotsc, n} and $X$ in $\cA$ and $Y$ in $\aC$.
\end{defn}

\begin{example}
If $\cC$ is a $\Sigma$-mapping algebra, then \w{\NlnC} is a
$\Sigma$-algebra of left cubical balls. If $\cC$ is an $\Omega$-mapping algebra, 
then \w{\NlnC} is an $\Omega$-algebra of left cubical balls.
\end{example}

The following Resolution Theorems now generalize those of Section \ref{cehor}:

\begin{thm}\label{thrt}
Let \w{\Tln} be a $\Sigma$-algebra of left cubical balls and let \w{A\sbu}
be an $\aC$-resolution of $X$ in $\cA$. Then there exists an
$n$-th order chain complex
$$
\xymatrix@1{K:\bZ(\infty, -1)_\otimes\sp{n}\ar[r] & \cT(n)}
$$
\noindent which is based on \w[.]{A\sbu} We call $K$ an 
\emph{$n$-th order resolution} of $X$ in \w[.]{\cT(n)}
\end{thm}

\begin{thm}\label{thdrt}
Let \w{\Tln} be an $\Omega$-algebra of left cubical balls and let \w{A\ubu}
be an $\aC$-coresolution of $Y$ in $\cA$. Then there exists an $n$-th order
chain complex
$$
L:\bZ(+1, -\infty)_\otimes\sp{n}~\to~\cT(n)
$$
which is based on \w[.]{A\ubu} We call $L$ an 
\emph{$n$-th order coresolution} of $Y$ in \w[.]{\cT(n)}
\end{thm}

\begin{proof}
The boundary property shows that there exists a functor
$$
\xymatrix@1{K'(1):\bZ(\infty, -1)_\otimes\sp{1}\ar[r] & \cT(1)}
$$
\noindent which satisfies the inclusion property and which is based on 
\w[.]{A\sbu} Hence by Theorem \ref{thota} we have a first-order chain 
complex \w{K(1)} which is based on \w[.]{A\sbu} Now the boundary property 
shows that there is a functor
$$
\xymatrix@1{K'(2):\bZ(\infty, -1)_\otimes\sp{1}\ar[r] & \cT(2)}
$$
\noindent which satisfies the inclusion property and which based on 
\w[.]{A\sbu} Again the boundary property shows there exists \w[,]{K'(3)} 
so that by Theorem \ref{thota} one obtains \w[.]{K(3)} Inductively, we thus 
have \w[.]{K = K(n)}
\end{proof}

\begin{example}
Let $\cC$ be the Eilenberg-Mac~Lane $\Omega$-mapping algebra. Then 
\emph{minimal} coresolutions \w{A\ubu} of $Y$ are defined in $\cA$ 
and hence we can find an $n$-th order coresolution of $Y$ in \w{\NlnC} 
based on \w[.]{A\ubu} We call
\w{\NlnC} the \emph{algebra of cohomology operations of order} \www[.]{\leq n+1} 
This is an $\Omega$-algebra of left cubical balls. It is convenient to consider 
the dual of \w[,]{\NlnC} which is a $\Sigma$-algebra of left cubical balls and
for which $\aC$ is the category of finitely generated free modules over 
the Steenrod algebra.
\end{example}

\begin{remark}
The main result of~\cite{BSe} computes the algebra of cohomology operations 
of order \www{\leq 2} in terms of a bigraded differential algebra $\dB$ over 
the ring \w[.]{\bZ/p\sp{2}} This leads to the conjecture that the 
algebra of cohomology operations of order \w{\leq n} \wb[,]{n \geq 1} can 
also be described up to equivalence by a bigraded differential algebra 
over \w[.]{\bZ/p\sp{2}}
\end{remark}

%
%
\sect{Higher order \ww{\Ext}-groups}
\label{cext}

In this section we deduce from higher order resolutions the associated 
higher order derived functors, which are higher order \ww{\Ext}-groups. 
We shall see that the $E\sb{n}$-term in the Adams spectral sequence is 
such a higher order \ww{\Ext}-group for \w[.]{n \geq 3}

It is classical that the \w{E\sb{2}} of the Adams spectral sequence is given 
by the `primary' \ww{\Ext}-groups of homological algebra, see~\wref[.]{eqbeg} 
In~\cite{BJSe} we studied the secondary \ww{\Ext}-groups which 
determine \w[.]{E\sb{3}}

\begin{defn}
Let \w{\Tln} \wb{n \geq 1} be a $\Sigma$-algebra of left cubical balls,
so for each \w{m = 1, \dotsc, n} we have the algebra of left $m$-cubical balls
\begin{myeq}
\cT(m) = (\cT(m), \cA, D\sb{m}, \OO\sp{m})
\end{myeq}
\noindent with \w{\aC \subset \cA} and 
\w{D\sb{m}(A, X) =\Hom\sb{\cA}(\Sigma\sp{m} A, X)} for objects $A$ in 
$\aC$ and $X$ in $\cA$. Let \w{A\sbu} be an $\aC$-resolution of 
$X$ in $\cA$ and let
\begin{myeq}\label{eqK}
\xymatrix@1{K:\bZ(\infty, -1)_\otimes\sp{n}\ar[r] & \cT(n)}
\end{myeq}
be an $n$-th order resolution of $X$ based on \w{A\sbu} 
(see Theorem \ref{thrt}). Furthermore,  let $Y$ be another object of $\cA$, 
and consider the diagram in $\cA$:
\begin{myeq}\label{eqdaeg}
\xymatrix{\dotsc\ar[r] & A\sb{r + m + 1}\ar[r] & A\sb{r + m}\ar[r] & 
\dotsc\ar[r] & A\sb{r}\ar[d]^\beta\ar[r]\sp{\delta\sb{r}} & \dotsc\ar[r] & 
A\sb{0}\ar[r] & X\\ 
& & & & Y}
\end{myeq}
The row of the diagram is the $\aC$-resolution \w{A\sbu} of $X$. We 
assume that $\beta$ is a cocycle, that is,
\begin{myeq}
\beta \delta\sb{r + 1} = 0.
\end{myeq}
\noindent Then $\beta$ represents an element \w{\{\beta\}} in the \ww{\Ext}-group
\begin{equation*}
\begin{split}
E\sb{2}\sp{r, 0} &= \Ext\sb{\cA}^r(X, Y)\\
&= \rH^r \Hom\sb{\cA}(A\sbu, Y)\\
&= \Ker \delta\sb{r + 1}\sp{\ast}/\image \delta\sb{r}\sp{\ast},
\end{split}
\end{equation*}
\noindent where
$$
\xymatrix@1{\delta\sb{r}\sp{\ast}:\Hom\sb{\cA}(A\sb{r - 1}, Y)\ar[r] & 
\Hom\sb{\cA}(A\sb{r}, Y)}~.
$$

Using the $\aC$-resolution \w{\Sigma\sp{s} A\sbu} of 
\w[,]{\Sigma\sp{s} X} we get accordingly for \w{s \geq 0} the bigraded 
\ww{\Ext}-group (see~\wref[),]{eqbeg}
$$
E\sb{2}\sp{r, s} = \Ext\sb{\cA}^r(\Sigma\sp{s} X, Y).
$$

In Definition \ref{dobstr} below we shall define a differential
\begin{myeq}
d\sb{2} = d\sb{2}\sp{r, s}:E\sb{2}\sp{r, s}~\to~ E\sb{2}\sp{r + 2, s + 1}~.
\end{myeq}

Moreover, inductively for \w{m \geq 2} we consider subquotients
\w{E\sb{m}\sp{r, s}} of \w[,]{E\sb{2}\sp{r, s}} together with differentials
\begin{myeq}
\xymatrix@1{d\sb{m} = d\sb{m}\sp{r, s}:E\sb{m}\sp{r, s}\ar[r] & 
E\sb{m}\sp{r + m, s + m - 1}}
\end{myeq}
\noindent satisfying \w[,]{d\sb{m} d\sb{m} = 0} and
$$
E\sb{m + 1}\sp{r, s} = \Ker(d\sb{m}\sp{r, s})/
\image(d\sb{m}\sp{r - m, s - m + 1}).
$$

We call \w{E\sb{m}\sp{r, 0}} for \w{m = 2, \dotsc, n + 1} the 
\emph{higher order \ww{\Ext}-groups} associated to the $n$-th order resolution 
$K$ of $X$ above. Replacing $X$ by \w[,]{\Sigma\sp{s} X} we obtain the groups 
\w[,]{E\sb{m}\sp{r, s}} accordingly.
\end{defn}

\begin{defn}\label{dobstr}
Let \w{\overline \beta \in E\sb{m + 1}\sp{r, 0}} be represented by
\w{\{\beta\} \in E\sb{2}\sp{r, 0}} \wb[,]{1 \leq m \leq n} and let $L$ be an
\wwb{m-1}-order chain complex
$$
\xymatrix@1{L:\bZ(\infty, r - 1)_\otimes\sp{m - 1}\ar[r] & \cT(m - 1)}
$$
based on the diagram
$$
\dotsc~\to~ A\sb{r + m + 1}~\dotsc~\to~ A\sb{r + 1}~\to~ A\sb{r}~\xra{\beta}~Y~,
$$
\noindent in $\cA$. Such an $L$ exists precisely because $\beta$ survives to the
\ww{E\sb{m + 1}}-term. We assume also that $L$ restricted to
\w{\bZ(\infty, r)} coincides with the \wwb{m - 1}skeleton of $K$ in 
\w[.]{\cT(m - 1)} The boundary property in \w{\cT(m)} shows that there 
is a functor
$$
\xymatrix@1{\hat L:\bZ(\infty, r - 1)_\otimes\sp{m}\ar[r] & \cT(m)}
$$
\noindent which is a pre-chain complex such that $\hat L$ restricted to
\w{\bZ(\infty, r)} coincides with the $m$-skeleton of $K$, and such that 
the \wwb{m - 1}skeleton \w{\hat L(m - 1)} of $\hat L$ satisfies
\w{q \hat L(m - 1) = L} in \w[.]{\cT(m - 1)} We then obtain the obstruction
$$
\OO \hat L(r + m + 1, \partial I\sb{m + 1}) = \OO\sb{B}(b\sb{1}, \dotsc, 
b\sb{m + 1}) \in \Hom\sb{\cA}(\Sigma\sp{m} A\sb{r + m + 1}, Y)~,
$$
where \w{B = T\sb{0}\sp{m}} and
$$
\hat L(r + m + 1, \partial I\sb{m + 1}) = 
\begin{cases}b\sb{1} = \hat L(r + m + 1, \emptyset \otimes I\sb{m}),\\ 
b\sb{k + 1} = 
\hat L(r + m + 1, I\sb{k} \otimes I\sb{m - 1}), & 1 \leq k \leq m - 1,\\ 
b\sb{m + 1} = \hat L(r + m + 1, I\sb{m} \otimes \emptyset)\end{cases}
$$
\noindent (see~\wref[).]{eqoe} The element
\w{\OO \hat L(r + m + 1, \partial I\sb{m + 1})} represents the value of the
differential
$$
d\sb{m + 1}\sp{r, 0}(\overline \beta) \in E\sb{m + 1}\sp{r + m + 1, m}~.
$$
\end{defn}

\begin{thm}\label{twelldef}
Let \w{\Tln=\NlnC} be the $\Sigma$-algebra of left cubical balls given by 
the complete $\Sigma$-mapping algebra $\cC$ of Example \ref{egspec}. Then 
Definition~\ref{dobstr} yields a well defined sequence of \ww{\Ext}-groups 
\w{E\sb{m}\sp{r, s}} for \w[.]{m = 2, \dotsc, n + 2} These groups depend on 
the weak equivalence class of the $\Sigma$-algebra of left cubical balls 
\w[,]{\Tln} and not on the choice of the $n$-th order resolution of $X$.
\end{thm}

\begin{remark}
Some additional assumptions on a $\Sigma$-algebra of left cubical balls 
\w{\Tln} are needed in Theorem \ref{twelldef} in order to guarantee
that the \emph{last} differential \w{d\sb{n+2}:E\sb{n+2}\to E\sb{n+2}} 
still satisfies \w[.]{d\sb{n+2}\circ d\sb{n+2}=0} 
The requisite assumptions are axiomatized in the Appendix, but they are certainly 
satisfied in Example \ref{egspec} (or any other case where we actually have a 
spectral sequence with these differentials).
\end{remark}

\begin{thm}\label{temma}
Let \w{\Tln=\NlnC} be the $\Sigma$-algebra of left cubical balls given by the
dual $\cC$ of the Eilenberg-Mac~Lane $\Omega$-algebra of left cubical balls.
Then the \ww{\Ext}-groups \w{E\sb{m}\sp{r, s}} \wb[,]{2 \leq m \leq n + 2}
yield the $m$-term \w{E\sb{m}} of the Adams spectral sequence which converges
to the stable homotopy set \w{\{Y, X\}} for finite spectra $X$ and $Y$.
\end{thm}

For \w{n = 1} this result is proved in \cite[Section 7]{BJSe}.
Theorems \ref{twelldef} and \ref{temma} are proved for all \w{n\geq 1} in the
following Section.

%
%
\sect{Strictification of higher order resolutions}
\label{cstrict}

In this section we show that any higher order resolution is actually equivalent
to a strict resolution, which we call its \emph{strictification}. This is
done by induction: the main difficulty is choosing the correct induction
hypotheses. Otherwise, the proof uses a long chain of standard arguments
in Quillen model categories.

It is well-known how to derive from such a strict resolution the
corresponding (Adams) spectral sequence, thus implying Theorems \ref{twelldef}
and \ref{temma}.

Throughout this section $\cC$ will be a $\Sigma$-mapping algebra
of the type described in Example \ref{egspec}, so the reader may take $\cC$ 
to be \w[.]{\Sp} 

\begin{notn}\label{nahat}
For such a $\cC$, we have the additive category \w[,]{\cA=\pi\sb{0}\cC} and
the full additive subcategory \w{\aC=\pi\sb{0}\{\dX\}} given by the class
of spectra $\dX$ in \ref{egspec}. Let $\haC$ denote the full subcategory 
of $\cA$ consisting of all objects $A$ in  which are isomorphic in $\cA$ to an 
object in $\aC$, with
\begin{myeq}\label{eqfs}
\aC~\subseteq~\haC ~\subseteq~\cA~.
\end{myeq}
\end{notn}

\begin{defn}
Let $\cT$ be an $n$-graded category (such as \w{\NnC} or  $\nnC$)
with a quotient functor \w[.]{q:\cT\sp{0}\to\cA} Let
\w{K,L:\bZ(\infty, -1)\sb{\otimes}\sp{n}\to\cT} be functors of $n$-graded 
categories. A \emph{weak equivalence} $\tau:K\to L$ over an object 
\w{X\in\cT\sp{0}} is a natural transformation $\tau$ which for objects 
$i$ in \w{\bZO} consists of a map
$$
\tau\sb{i}:K\sb{i}\to L\sb{i} \hsm\text{in}\hsm \cT\sp{0}
$$
\noindent which induces an isomorphism \w{q\tau\sb{i}} in $\cA$.
For \w[,]{i=-1} the map \w{\tau\sb{-1}:K\sb{-1}=X=L\sb{-1}} is the identity of $X$.
For a morphism \w{V:i\to j} in \w[,]{\bZO} we have the commutative
diagram in $\cT$:
$$
\xymatrix{
K\sb{i} \ar[r]\sp{\tau\sb{i}}\ar[d]\sb{K(V)}  & L\sb{i} \ar[d]\sp{L(V)}  \\
K\sb{j} \ar[r]\sp{\tau\sb{j}}  & L\sb{j}
}
$$
\noindent or equivalently, \w[.]{\tau\sb{j}K(V)=L(V)\tau\sb{i}}
Let $\sim$ be the equivalence relation generated by weak equivalences over $X$.
\end{defn}

\begin{lemma}
Let \w{K,L:\bZ(\infty, -1)\sb{\otimes}\sp{n}\to\NnC} be $n$-th order
resolutions of $X$. If \w{K\sim L} are weakly equivalent over $X$, then
the higher \ww{\Ext}-groups defined by $K$ and $L$ are
isomorphic.\hfill$\Box$
\end{lemma}

We shall show that the higher \ww{\Ext}-groups actually do not depend on the
choice of resolution of $X$. For this, we use the strictification of
resolutions.

\begin{defn}
Let \w{K:\bZ(\infty, -1)\sb{\otimes}\sp{n}\to\cT} be a functor of
$n$-graded categories, and \w[.]{N\geq 0}  We say that $K$ is $N$-\emph{strict}
if for all \w{i\leq N} and \w{k=1,\dotsc,n} we have \w[.]{K(i,I\sb{k})=o} 
In this case \w{\delta\sb{i}=K(i,\emptyset):K\sb{i}\to K\sb{i-1}} yields 
a sequence of maps 
$$
K\sb{N}~\xrightarrow{\delta}~K\sb{N-1}~\xrightarrow{\delta}~
\dotsc~\xrightarrow{\delta}~K\sb{0}~\xrightarrow{\delta}~K\sb{-1}
$$
\noindent in \w[,]{\cT\sp{0}} with \w{K\sb{-1}=X} and \w[.]{\delta\delta=o} 

When $\cT$ is \w{\NnC} or \w{\nnC} for a topologically enriched model 
category $\cC$ as above, we say that $K$ as above is $N$-\emph{fibrant} 
if there are fiber sequences
$$
Z\sb{i}~\xrightarrow{j\sb{i}}~K\sb{i}~\stackrel{p\sb{i}}{\epic}~Z\sb{i-1}
\hspace*{9mm} p\sb{i}\circ j\sb{i}=o~,
$$
\noindent in the model category $\cC$ with \w{\delta\sb{i}=j\sb{i-1}p\sb{i}} 
for each \w[,]{0\leq i<N} and \w{\delta\sb{N}} admits a factorization
$$
K\sb{N}~\xrightarrow{\ovp\sb{N}}~Z\sb{N-1}~
\xrightarrow{j\sb{N-1}}~K\sb{N-1}~.
$$
\noindent Note that \w{p\sb{i}:K\sb{i}\to Z\sb{i-1}} above is a fibration, 
but \w[,]{\ovp\sb{N}:K\sb{N}\to Z\sb{N-1}} need not be a fibration. 
More generally, throughout this section we use $p$ for fibrations, while 
$\ovp$ need not be a fibration.

Moreover, $K$ is $N$-\emph{exact} if for each \w{0\leq i<N} and $A$ in $\aC$, 
the induced sequence
\begin{myeq}\label{eqseshtpy}
\Hom\sb{\cA}(A,Z\sb{i})~\to~\Hom\sb{\cA}(A,K\sb{i})~\to~\Hom\sb{\cA}(A,Z\sb{i-1})
\end{myeq}
\noindent is a short exact sequence of abelian groups.
\end{defn}

\begin{thm}\label{thstrict}
Let \w{n\geq 1} and \w[,]{N\geq 0} and let
\w{K:\bZ(\infty, -1)\sb{\otimes}\sp{n}\to\NnC} be an $n$-th order resolution
of $X$ based on the $\aC$-resolution \w{A\sbu}
of $X$ in $\cA$. Then there exists an $N$-strict $N$-fibrant $N$-exact
$n$-th order resolution $L$ of $X$ based on an $\haC$-resolution
\w{\wA\sbu} of $X$ in $\cA$ such that \w{L\sim K} are
weakly equivalent over $X$.
\end{thm}

\begin{remark}\label{rstrict}
The dual of Theorem \ref{thstrict} holds for coresolutions.
\end{remark}

\begin{notn}
Here we use the bigger category $\haC$ of \S \ref{nahat}.
The resolutions \w{A\sbu} and \w{\wA\sbu} yield by the
weak equivalence \w{L\sim K} over $X$ the commutative diagram in $\cA$:
$$
\xymatrix{
\dotsc \ar[r]& A\sb{1} \ar[r]\sp{\delta}\ar[d]\sb{\cong}  &
A\sb{0} \ar[d]\sp{\cong} \ar[r]\sp{\delta} & X \ar[d]\sp{=} \\
\dotsc \ar[r]& \wA\sb{1} \ar[r]\sp{\delta} &
\wA\sb{0} \ar[r]\sp{\delta} & X~,
}
$$
\noindent where the vertical arrows are isomorphisms in $\cA$ and we have
\w{A\sb{i}=K\sb{i}} and \w{\wA\sb{i}=L\sb{i}} for \w[.]{i\geq -1}

A $\Sigma$-mapping algebra $\cC$ is given by an underlying model category
and cubes
$$
I\sp{n}~\lra~ \Mor\sb{\cC}(X,Y)~,
$$
\noindent having an adjoint
\begin{myeq}\label{eqhs}
(I\sp{n}\times X)/(I\sp{n}\times \ast)~\lra~Y~,
\end{myeq}
\noindent where \w{(I\times X)/(I\times \ast)} is the \emph{pointed cylinder} 
on $X$.
\end{notn}

\begin{lemma}\label{lnone}
Theorem \ref{thstrict} holds for \w[.]{n=1}
\end{lemma}

\begin{proof}
By induction on \w{m\geq 0} we construct a sequence of first-order resolutions
\w{L\sp{(m)}} of $X$ based on \w[,]{A\sbu} such that:

\begin{enumerate}
\renewcommand{\labelenumi}{(\roman{enumi})~}
\item \w{L\sp{(0)}} is the original first-order resolution $K$.
\item For each \w[,]{m\geq 1} \w{L\sp{(m)}} is $m$-strict, $m$-fibrant, $m$-exact
and weakly equivalent over $X$ to \w{L\sp{(m-1)}} (and thus by induction to $K$).
\item For each \w[,]{m\geq 0} \w{L\sp{(m+1)}\sb{i}=L\sp{(m)}\sb{i}} for 
all \w{i\neq m} (so each new approximation differs from the previous 
one in a single dimension). 
\end{enumerate}

Assuming by induction that for some \w[,]{N\geq 0} we have such a $N$-strict, 
$N$-fibrant, and $n$-exact \w[,]{K':=L\sp{(N)}}  we construct 
\w{L':=L\sp{(N+1)}} as follows:

By assumption, we have a factorization of 
\w{\delta\sb{N-1}:K'\sb{N-1}\to K'\sb{N-2}} as
$$
K'\sb{N-1}~\stackrel{p\sb{N-1}}{\epic}~Z\sb{N-2}~
\xrightarrow{j\sb{N-2}}~K'\sb{N-2}~,
$$
\noindent with \w{Z\sb{N-1}} the fiber of the fibration \w{p\sb{N-1}}
(and thus the kernel of \w[,]{\delta\sb{N-1}} since \w{j\sb{N-2}} is monic).
Because \w{K'} is $N$-strict, \w[,]{\delta\sb{N-1}\circ\delta\sb{N}=o} so
\w{\delta\sb{N}:K'\sb{N}\to K'\sb{N-1}} factors through 
\w{Z\sb{N-1}=\Ker\delta\sb{N-1}} as 
$$
K'\sb{N}~\xra{\ovp\sb{N}}~Z\sb{N-1}~\xrightarrow{j\sb{N-1}}~K'\sb{N-1}~,
$$
\noindent and in the model category $\cC$ we may factor \w{\ovp\sb{N}} as
\begin{myeq}\label{eqfact}
\xymatrix{K'\sb{N}\hspace*{2mm} 
\ar@{>->}[r]\sp{i\sb{N}}\sb{\simeq} & ~L'\sb{N}~\ar@{->>}[r]\sp<<<<{p\sb{N}} &
~Z\sb{N-1}
}
\end{myeq}
\noindent where \w{p\sb{N}} is a fibration and \w{i\sb{N}} is a trivial 
cofibration. 

Hence we have a commuting diagram
\begin{myeq}\label{eqhtpy}
\xymatrix{
& K'\sb{N+1} \ar[d]\sb{\delta\sb{N+1}} \ar@/^{4pc}/[dddd]\sp{o} & \\
L'\sb{N}\ar@{->>}[dd]\sb{p\sb{N}}  & 
K'\sb{N} \ar[l]\sb{i\sb{N}}\sp{\sim}\ar[dd]\sb{\ovp\sb{N}} & \\
& \ar@{=>}[r]\sp{H} & \\
Z\sb{N-1}\ar[d]\sb{j\sb{N-1}}  & Z\sb{N-1} \ar@{=}[l] \ar[d]\sb{j\sb{N-1}} & \\
L'\sb{N-1} & K'\sb{N-1} \ar@{=}[l] &
}
\end{myeq}
\noindent where we set \w{L'\sb{i}:=K'\sb{i}} for all \w[,]{i<N} and \w{L'\sb{N}} 
is defined by \wref[.]{eqfact}  In fact, we shall also set 
\w{L'\sb{i}:=K'\sb{i}} for all \w[,]{i>N} by (iii), so \w{L'} differs from 
\w{K'} only at the $N$-th slot.

Since \w[,]{\delta\sb{N-1}\circ\delta\sb{N}=o} 
\w[,]{\delta\sb{N-1}=j\sb{N-2}\circ p\sb{N-1}} and \w{j\sb{N-2}} is monic,
in fact \w[,]{p\sb{N-1}\circ\delta\sb{N}=o} so \w{p\sb{N-1}H} 
is a self-nullhomotopy \w[,]{o\Rightarrow o:K'\sb{N+1}\to Z\sb{N-2}} 
and thus induces a map \w[,]{\alpha:\Sigma K'\sb{N+1}\to Z\sb{N-2}} 
with \w{j\sb{N-2}\alpha:\Sigma K'\sb{N+1}\to K'\sb{N-2}} being (the adjoint of)
the obstruction element \w{\OO K'(N+1, \partial I\sb{2})} of \S \ref{dnocc}
(for \w[),]{n=1} so it is nullhomotopic.
Therefore, by \wref{eqseshtpy} in the definition of $N$-exactness, 
$\alpha$ itself has a nullhomotopy
\w[.]{G:\alpha\Rightarrow o:\Sigma K'\sb{N+1}\to Z\sb{N-2}}

Expressing $G$ and $H$ in terms of maps out of products with cubes 
(or cubical horns) we obtain the solid commutative diagram of 
Figure \ref{fig4}, in which \w{H:\delta\sb{N}\circ\delta\sb{N+1}\Rightarrow o} 
is extended from the left vertical map in the upper horn by constant
homotopies along the horizontal edges of the square (with \w{\id\sb{f}} 
abbreviated to $f$).  When the resulting map is post-composed with the fibration
\w[,]{p\sb{N-1}} we can fill in the square by the homotopy 
of homotopies \w[,]{G:\alpha\Rightarrow o} with the maps 
\w{K'\sb{N+1}\times I\to Z\sb{N-1}}indicated along the edges of the square.

%
%
\begin{figure}[htbp]
\begin{picture}(190,140)(0,0)
%
%
\put(-3,100){$K'\sb{N+1}~\times$}
\put(40,90){\line(1,0){30}}
\put(40,90){\line(0,1){30}}
\put(40,120){\line(1,0){30}}
\put(50,106){\vector(1,0){105}}
\put(115,108){{\scriptsize $H$}}
\bezier{400}(70,125)(100,140)(150,115)
\put(115,130){{\scriptsize $\delta\sb{N}\circ\delta\sb{N+1}$}}
\put(150,115){\vector(3,-1){10}}
\bezier{400}(70,85)(100,70)(150,95)
\put(105,83){{\scriptsize $o$}}
\put(150,95){\vector(3,1){10}}
\put(57.5,80){\oval(5,5)[t]}
\put(55,80){\vector(0,-1){32}}
\put(0,60){{\scriptsize $\id\sb{K'\sb{N+1}}\times\inc$}}
\put(60,60){{\scriptsize $\simeq$}}
%
%
\put(-3,20){$K'\sb{n+1}~\times$}
\put(40,10){\line(1,0){30}}
\put(40,10){\line(0,1){30}}
\put(70,10){\line(0,1){30}}
\put(40,40){\line(1,0){30}}
%
%
\multiput(66,12)(2,2){2}{\circle*{0.5}}
\multiput(62,12)(2,2){4}{\circle*{0.5}}
\multiput(58,12)(2,2){6}{\circle*{0.5}}
\multiput(54,12)(2,2){8}{\circle*{0.5}}
\multiput(50,12)(2,2){10}{\circle*{0.5}}
\multiput(46,12)(2,2){12}{\circle*{0.5}}
\multiput(42,12)(2,2){14}{\circle*{0.5}}
\multiput(42,16)(2,2){2}{\circle*{0.5}}
\multiput(50,24)(2,2){8}{\circle*{0.5}}
\multiput(48,26)(2,2){7}{\circle*{0.5}}
\multiput(46,28)(2,2){6}{\circle*{0.5}}
\multiput(42,28)(2,2){6}{\circle*{0.5}}
\multiput(42,32)(2,2){4}{\circle*{0.5}}
\multiput(42,36)(2,2){2}{\circle*{0.5}}
\put(42,21){{\scriptsize $\alpha$}}
\put(55,42){{\scriptsize $o$}}
\put(55,4){{\scriptsize $o$}}
\put(72,21){{\scriptsize $o$}}
\put(80,26){\vector(1,0){75}}
\put(115,17){{\scriptsize $G$}}
\multiput(77,33)(3,2){28}{\circle*{1.0}}
\put(160,88){\vector(3,2){5}}
\put(120,52){{\scriptsize $\widehat{G}$}}
%
%
\put(160,100){$K'\sb{N-1}$}
\put(176,95){\vector(0,-1){60}}
\put(176,92){\vector(0,-1){60}}
\put(180,60){{\scriptsize $p\sb{N-1}$}}
\put(160,20){$Z\sb{N-2}$}
\end{picture}

\caption[fig4]{Lifting nullhomotopies}
\label{fig4}
\end{figure}
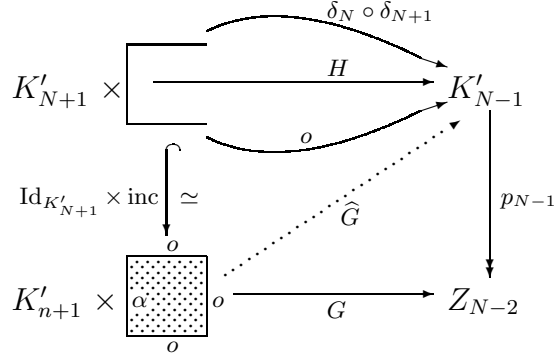

Since the left vertical inclusion is a trivial cofibration, we can use the 
lifting property in the model category $\cC$ to obtain $\widehat{G}$, whose
restriction to the right vertical edge of the lower square yields a new
nullhomotopy \w{H':K'\sb{N+1}\to K'\sb{N-1}} of 
\w[,]{\delta\sb{N}\circ\delta\sb{N+1}} which is homotopic to $H$ and 
thus represents the same track. Since \w[,]{p\sb{N-1}H'=\id\sb{o}} \w{H'} 
factors through the kernel \w{Z\sb{N-1}} of \w{p\sb{N-1}} as 
\w[.]{H'=j\sb{N-1}\ovH}

We then obtain the following diagram, where \w{IK'\sb{N+1}} is the cylinder,
with inclusions \w{i\sb{0}} and \w{i\sb{1}} (setting 
\w[).]{L'\sb{N+1}:=K'\sb{N+1}}
\begin{myeq}\label{eqth}
\xymatrix{
& L'\sb{N+1}~ \ar@{.>}[d]\sp{\delta\sb{N+1}\sp{L'}} 
\ar@/_{4pc}/[dd]\sb{o}
\ar@{^{(}->}[r]\sp{i\sb{0}}\sb{\simeq} &
IK'\sb{N+1} \ar[ld]^<<<{\widehat{H}} &
~K'\sb{N+1} \ar[d] \ar@/^{4pc}/[dd]\sp{o} 
\ar@{_{(}->}[l]\sb{i\sb{1}}\sp{\simeq} & \\
& L'\sb{N}\ar@{->>}[d]\sp{p\sb{N}\sp{L'}} \ar@{=}[l] & &
~K'\sb{N} \ar[d] \ar@{_{(}->}[ll]\sb{i\sb{N}}\sp{\simeq} \ar@{=>}[r]\sp{\ovH}& \\
& Z\sb{N-1} &&  Z\sb{N-1} \ar@{=}[ll] &
}
\end{myeq}
Here $\widehat{H}$ is a lift of $\ovH$ through \w[,]{p\sb{N}} so that
the diagram commutes with \w[.]{p\sb{N}\sp{L'}\widehat{H}i\sb{0}=o} Hence for
\w[,]{\delta\sb{N+1}\sp{L'}:=\widehat{H}i\sb{0}} the left hand side 
is \wwb{N+1}-strict and \wwb{N+1}-fibrant. Moreover, the left hand side 
is \wwb{N+1}exact, that is:
\begin{myeq}\label{eqmhom}
(p\sb{N}\sp{L'})\sb{\ast}~=~\Hom\sb{\cA}(A,p\sb{N})~:~\Hom\sb{\cA}(A,L'\sb{N})~
\to~\Hom\sb{\cA}(A,Z\sb{N-1})
\end{myeq}
is surjective for all $A$ in $\aC$.

To show this, given \w{\alpha\in\Hom\sb{\cA}(A,Z\sb{N-1})} we have
\w[,]{(\delta\sb{N-1} j\sb{N-1})\sb{\ast}\alpha=0} so that by exactness 
of \w{A\sbu}
we have
$$
(j\sb{N-1})\sb{\ast}\alpha~=~(\delta\sb{N}\sp{L'})\sb{\ast}\beta~=~
(j\sb{N-1}p\sp{L'}\sb{N})\sb{\ast}\beta
$$
\noindent for some \w[.]{\beta\in\Hom\sb{\cA}(A,L'\sb{N})}
Thus \w{\alpha=(p\sp{L'}\sb{N})\sb{\ast}\beta} by injectivity of 
\w[.]{(j\sb{N-1})\sb{\ast}}

The fiber sequence for \w{p\sb{N}\sp{L'}} thus has
\w{\Hom\sb{\cA}(A,\Omega p\sb{N})=\Hom\sb{\cA}(\Sigma A,p\sb{N})}
surjective, since \w[,]{\Sigma A\in\aC} which implies that 
\w{(j\sb{N})\sb{\ast}} is one-to-one.

We now construct weak equivalences
\begin{myeq}\label{eqta}
L'~\xrightarrow{i\sb{0}}~ R \xleftarrow{i\sb{1}}~K'
\end{myeq}
where \w{i\sb{0}} and \w{i\sb{1}} are the identity in degrees $<N$.
In degree $N$ the resolution \w{L'} is given by the diagram above.
In dimension $0$, diagram \wref{eqta} is given by the commutative diagram:

\mydiagram[\label{eqrc}]{
\dotsc \ar[r] & K'\sb{N+2} \ar[r] \ar[d] & K'\sb{N+1} \ar[r] \ar[d] &
K'\sb{N} \ar[r] \ar@{>->}[d]\sp{\sim} & K'\sb{N-1} \ar[r] \ar@{=}[d] & \dotsc\\
\dotsc \ar[r]\sp{I\delta\sb{N+3}} & IK'\sb{N+2} \ar[r]\sp{I\delta\sb{N+2}} &
IK'\sb{N+1} \ar[r]\sp{\widehat{H}} & L'\sb{N} \ar@{=}[d] \ar[r] & 
K\sb{N-1} \ar[r] \ar@{=}[d] & \dotsc\\
\dotsc \ar[r] & K'\sb{N+2} \ar[r] \ar[u] & K'\sb{N+1} \ar[r] \ar[u] &
L'\sb{N} \ar[r]  & L'\sb{N-1} \ar[r] & \dotsc
}
\noindent It is easy to find appropriate \w[,]{R(N+1,I\sb{1})} 
\w[,]{R(N+2,I\sb{1})} and \w{R(M,I\sb{1})=IK'(M,I\sb{1})} for \w[,]{M\geq N+3} 
so that \wref{eqta} is well defined. Here we use the adjoint maps in 
\wref[.]{eqhs} This completes the proof of Theorem \ref{thstrict} for \w[.]{n=1}
\end{proof}

\begin{tlemma}\label{ltransport}
Let $K$ be an $n$-th order resolution of $X$ in \w{\NnC} with \wwb{n-1}skeleton
\w[.]{K\sp{(n-1)}} Let
$$
L\sp{(n-1)}~\xrightarrow{f}~K\sp{(n-1)}~\xrightarrow{g}~R\sp{(n-1)}~
\hspace*{9mm}\text{in}\hspace*{2mm} \Nul{n-1}\cC
$$
\noindent be weak equivalences over $X$. Then there exist unique $n$-th order
resolutions $L$ and $R$ in \w{\NnC} together with weak equivalences
$$
L~\xrightarrow{\tilde{f}}~K~\xrightarrow{\tilde{g}}~R
$$
\noindent which, restricted to \wwb{n-1}skeleta, coincide with
$f$ and $g$ respectively.
\end{tlemma}

\begin{proof}
Given a map \w{u:\partial I\sp{n}\to U} in \w[,]{\TT} let \w{[I\sp{n},U]\sp{u}}
denote the set of homotopy classes of extensions of $u$ along the inclusion
\w{\partial I\sp{n}\hra I\sp{n}} (relative to \w[).]{\partial I\sp{n}}
By \cite[Proposition II.2.11]{BAl}, any weak equivalence \w{\hat{g}:U\to V} induces
a bijection \w[.]{\hat{g}\sb{\ast}:[I\sp{n},U]\sp{u}\to[I\sp{n},V]\sp{\hat{g}u}}
Note that given a map of \wwb{n-1}th order chain complexes
\w[,]{g:K\sp{(n-1)}\to R\sp{(n-1)}} an extension of \w{R\sp{(n-1)}}
to an $n$-th order resolution $R$ is determined by a choice of $n$-tracks for 
$R$, and these are precisely elements in sets of the form 
\w{[I\sp{n},V]\sp{\hat{g}u}} for various mapping spaces $U$ for $K$ and $V$ 
for $L$. Therefore, \w{g\sb{\ast}} induces a unique $n$-th order chain 
complex structure $R$ extending \w[.]{R\sp{(n-1)}} 

Similarly for \w[.]{f:L\sp{(n-1)}\to K\sp{(n-1)}}
\end{proof}

\begin{lemma}\label{lquotient}
Let $q:\nnC\to\NnC$ be the quotient map of Section \ref{cncatlcs}
and let \w{qK} and \w{qL} be $n$-th order resolutions of $X$ in \w{\NnC}.
If \w{qK\sim qL} are weakly equivalent, then also \w{K\sim L} are weakly 
equivalent over $X$ in \w[.]{\nnC}\hfill $\Box$
\end{lemma}

\begin{lemma}
\label{lstrictify}
Let $K$ be an $n$-th order resolution of $X$ in \w{\NnC} for some 
\w[,]{n\geq 2} and assume that its \wwb{n-1}skeleton \w{K\sp{(n-1)}} 
(in \w[)]{\nnC} is \wwb{N-1}(strict, fibrant, exact), for some 
\w[.]{N\geq 1} Then there is a weakly equivalent $n$-th order resolution 
\w{L\sim K} of $X$ which is $N$-(strict, fibrant, exact).
\end{lemma}

\begin{proof}
Since the \wwb{n-1}skeleton of $K$ is strict, the $n$-track \w{K(i,I\sb{n})} 
is given by elements
$$
\alpha\sb{i}~\in~\Hom\sb{\cA}(\Sigma\sp{n}K\sb{i},K\sb{i-n-1})\hspace*{9mm} 
i-n-1\geq -1
$$
\noindent for each \w[.]{2\leq i\leq N} We shall show by induction on 
\w{0< i\leq N+1} that $K$ is weakly equivalent to an $n$-th order resolution 
which is \wwb{i-1}(strict, fibrant, exact), so \w{\alpha\sb{j}=\omi} 
for \w[.]{j<i} 

By the obstruction property of $K$ we have
\w[.]{\delta\alpha\sb{i}\pm\alpha\sb{i-1}\delta=\omi}
Then \w[,]{\delta\alpha\sb{i}=\omi} and the exactness yields $\beta$ with
\w[.]{\alpha\sb{i}=\beta\delta}
We construct weak equivalences \w{L\to R\leftarrow K} in \w{\NnC} which in
dimension $0$ are given by the commutative diagram
$$
\xymatrix{
\dotsc \ar[r] & K\sb{i+1} \ar[r] \ar[d] & K\sb{i} \ar[r] \ar[d] &
K\sb{i-1} \ar[r] \ar@{=}[d] &\dotsc &, & K\ar[d]\sp{i\sb{0}}\\
\dotsc \ar[r]\sp{I\delta} & IK\sb{i+1} \ar[r] &  IK\sb{i} \ar[r] & 
K\sb{i-1} \ar@{=}[d] \ar[r]
&\dotsc &, & R\\
\dotsc \ar[r] & K\sb{i+1} \ar[r] \ar[u] & K\sb{i} \ar[r] \ar[u] &
K\sb{i-1} \ar[r]  &\dotsc &, & L\ar[u]\sb{i\sb{1}}\\
}
$$
\noindent The \wwb{n-2}skeleton of $R$ is strict. We define \w{R(i,I\sb{n-1})}
by $\beta$.
Then we can choose\w{R(i,I\sb{n})} such that \w{i\sb{0}} is a well-defined map
and $L$ is $i$-strict in \w{\NnC}.
\end{proof}

\begin{proof}[Proof of Theorem \protect\ref{thstrict} for \ $n\geq 2$]
By induction on $n$, we assume that the Theorem holds for \w[.]{n-1} Let $K$ be 
a resolution of $X$ in \w[,]{\NnC} and let $K\sp{(n-1)}$ be the $(n-1)$-skeleton
of $K$ in \w[.]{\nnC} For \w{qK\sp{(n-1)}} we get by assumption a weak equivalence
\w[,]{qK\sp{(n-1)}\sim qL\sp{(n-1)}} where $L$ is $N$-strict. Hence by
Lemma \ref{lquotient} we have \w[,]{K\sp{(n-1)}\sim L\sp{(n-1)}} and by the
Transport Lemma \ref{ltransport}  we get \w{K\sim L} in \w[,]{\NnC} where
\w{L\sp{(n-1)}} is strict.
Now Lemma \ref{lstrictify} yields \w{L\sim L'} in \w[,]{\NnC} where \w{L'}
is $N$-strict.
\end{proof}

\begin{proof}[Proof of Theorem \protect\ref{twelldef}]
Let $K$ and $L$ be two resolutions of $X$ in \w[.]{\NnC} By Theorem
\ref{thstrict} we have \w{L\sim L'} and \w[,]{K\sim K'} where \w{L'} and \w{K'}
are $N$-(strict, fibrant, exact) for large $N$. This yields a map of spectral
sequences \w{E\sb{K}\to E\sb{L}} which induces an isomorphism on the 
$E\sb{2}$-term. Hence \w{E\sb{K}\to E\sb{L}} is also an isomorphism.
\end{proof}

\begin{proof}[Proof of Theorem \protect\ref{temma}]
An $N$-strict $N$-fibrant $N$-exact coresolution of $X$ for large $N$, as in
Remark \ref{rstrict}, corresponds to the $X$-coaugmented sequence in
\cite[\S 6.7]{BJSe} given by the Adams fiber tower \cite[7.1]{BJSe}.
\end{proof}

\begin{remark}
Strictification results for $\infty$-homotopy commutative diagrams appear in
\cite[Theorem~IV.4.37]{BVogHI} and \cite[Theorem~2.4]{DKSmH}, inter alia.
However, these do not yield the precise case needed for Theorem \ref{thstrict}.
The explicit construction given in this context may be of independent interest.
\end{remark}

%
%
\sect{The secondary \ww{Ext}-group}
\label{cddt}

We now illustrate the general theory described above in the first interesting
case, namely, the secondary \ww{\Ext}-group which is the \ww{E\sb{3}}-term
of the Adams spectral sequence. This is determined by the
\ww{d\sb{2}}-differential, which we now describe. The actual computation of this
differential appears in \cite{BSe,BJSe,BJiblD}.  As a test for the theory,
computer calculations based on those results were carried out, recovering the
values of \w{d\sb{2}} in \cite[App.\ 3]{RavC}.

We first specialize Definition \ref{dobstr} to the case \w[:]{m=2}
let \w{\cT(\leq 1)} be a $\Sigma$-algebra of left cubical balls and let
$K$ be a resolution in \w{\cT(1)} of $X$, based on \w{A\sbu} in $\cA$,
see~\wref[.]{eqK} Then
$$
d\sb{2}:\Ext\sb{\cA}^r(X, Y)~\to~ \Ext\sb{\cA}\sp{r + 2}(\Sigma X, Y)
$$
\noindent is given as follows:

For \w{\{\beta\} \in \Ext\sb{\cA}^r(X, Y)} with \w{\beta:A\sb{r}\to Y} in $\cA$,
we have \w[,]{\beta \delta\sb{r + 1} = 0} so that there is a $1$-track $H$ with
\w[.]{\partial\sp{1} H = \beta \delta\sb{r + 1}} On the other hand $K$ yields a
$1$-track $G$ with \w[.]{\partial\sp{1} G = \delta\sb{r + 1} \delta\sb{r + 2}}
Then the obstruction of Definition \ref{dobstr} is
\begin{myeq}
\omega~=~\OO \hat L(r + 2, \partial I\sb{2})~=~\OO(H \delta\sb{r + 2}, \beta G)
\in [\Sigma A\sb{r + 2}, Y]~,
\end{myeq}
\noindent and this element represents \w[.]{d\sb{2}\{\beta\} = \{\omega\}}

\begin{lemma}
\label{lem:d2WellDefined}
The differential \w{d\sb{2}} is well defined.
\end{lemma}

\begin{proof}
We first check that $\omega$ is a cocycle. In fact,
\begin{align*}
\omega(\Sigma \delta\sb{r + 3}) &
= \OO(H \delta\sb{r + 2}, \beta G)(\Sigma \delta\sb{r + 3})\\
& = \OO(H \delta\sb{r + 2} \delta\sb{r + 3}, \beta G \delta\sb{r + 3})\tag{1}\\
& = \OO(H \delta\sb{r + 2} \delta\sb{r + 3}, \beta \delta\sb{r + 1} G')\tag{2}\\
& = 0.\tag{3}
\end{align*}
\noindent Here (1) holds by naturality of $\OO$ (see \wref[).]{eqnatur}
Moreover, \w{G'} in (2) is the $1$-track with
\w{\partial\sp{1} G' = \delta\sb{r + 2} \delta\sb{r + 3}} given by the 
resolution $K$, so that \w[.]{\OO(G \delta\sb{r + 3}, \delta\sb{r + 1} G') = 0} 
Hence by naturality also 
\w[,]{\OO(\beta G \delta\sb{r + 3}, \beta \delta\sb{r + 1} G') = 0} so that (2)
holds by the complement rule in Lemma \ref{lcompl}. Finally (3) holds by the
triviality rule.

Next, we show that \w{\{\omega\}} does not depend on the choice of $H$.
If we choose \w{H'} instead, there is an $\alpha$ with \w[,]{H' = H + \alpha}
and we get
$$
\omega' = \OO((H + \alpha) \delta\sb{r + 2}, \beta G) =
\omega \pm \alpha \delta\sb{r + 2}
$$
\noindent by the action rule. Hence \w{\omega - \omega'} is a coboundary, so that
\w[.]{\{\omega'\} = \{\omega\}}

Finally, we check that \w{d\sb{2}\{\beta\}} is trivial if $\beta$ is a coboundary
\wh that is,  \w[.]{\beta = \beta' \delta\sb{r}} In fact, we can then choose $H$
to be the $1$-track \w[,]{\beta' G''} where \w{G''} with
\w{\delta\sp{1} G'' = \delta\sb{r} \delta\sb{r + 1}} is given by $K$, so that
\w[.]{\OO(G'' \delta\sb{r + 2}, \delta\sb{r} G) = 0} Hence also
\w[,]{\OO(\beta' G'' \delta\sb{r + 2}, \beta' \delta\sb{r} G) = 0} so that
\w[.]{\OO(H \delta\sb{r + 2}, \beta G) = 0}
\end{proof}

The Lemma is proved in \cite{BJSe} in the context of track categories, above
we use only algebras of left $1$-cubical balls. The proof that 
\w{d\sb{2} d\sb{2} = 0} requires the product rule below.

Next, we prove that the assumption on $L$ in Definition \ref{dobstr} is
satisfied for \w[.]{m = 2} This leads to the definition of the differential
\w[.]{d\sb{3}}

\begin{lemma}
Suppose that \w[.]{d\sb{2}\{\beta\} = 0} Then for \w{m = 2} there is a chain 
complex $L$ as in Definition \ref{dobstr}.
\end{lemma}

\begin{proof}
The assumption \w{d\sb{2}\{\beta\} = 0} shows that
\w{\omega = \OO(H \delta\sb{r + 2}, \beta G) = \alpha \delta\sb{r + 2}} is a
coboundary. Hence we get by the action rule \w[,]{H'=H\pm\alpha} so that
\w[.]{\OO(H'\delta\sb{r + 2},\beta G)=0} Hence we define the chain complex $L$
by \w{H'} and by $K$.
\end{proof}

In the context of a $\Sigma$-track algebra \w{\Tln} \wb[,]{n \geq 1} the
following result can be proved which is the higher dimensional analogue of
Lemma~\ref{lem:d2WellDefined}.

\begin{prop}
Given $L$, $\hat L$, and
$$
\omega = \OO \hat L(r + m + 1, \partial I\sb{m + 1})
$$
\noindent as in Definition \ref{dobstr}, then $\omega$ is a cocycle, that is:
$$
\omega(\Sigma\sp{m} \delta\sb{r + m + 2})= 0~.
$$
\noindent Moreover, if $\beta$ is a coboundary, then $L$ and $\hat L$ can be
chosen so that \w[.]{\omega = 0} Let $L$ be given and let $\hat L$ and 
$\overline L$ be two choices as in Definition \ref{dobstr}. Then
\w{\overline{\omega}:=\OO\overline{L}(r + m + 1, \partial I\sb{m + 1})} and
$\omega$ differ by a coboundary \wh that is,
\w[.]{\omega - \overline{\omega} = \alpha (\Sigma\sp{m}\delta\sb{r + m + 1})}
\end{prop}

\setcounter{section}{18}
%
%
\secta{Appendix: Complete algebras of left cubical balls}
\label{ccta}

Definition  \ref{dobstr} of the differential \w{d\sb{m+1}} makes sense in
any $\Sigma$-algebra of left cubical balls \w[,]{\Tln} but we need not have
\w[.]{d\sb{m+1}\circ d\sb{m+1}=0} However, this is automatically
satisfied in the topological examples of \S \ref{egspec} (and the differential
in the Adams spectral sequence is indeed given by Definition \ref{dobstr}).

In this appendix, which is not needed for the rest of the paper, we describe
(without proof) some properties of $\Sigma$-algebras of left cubical balls
which are needed to define the higher \ww{\Ext}-groups.  For this purpose
we introduce the notion of a \emph{complete} $\Sigma$-algebra of left cubical
balls. An example is provided by
the $\Sigma$-algebra of left cubical balls \w[,]{\NlnC} where $\cC$ is a
complete $\Sigma$-mapping algebra as in Definition \ref{dsmac}.

For such a $\cC$, the endofunctor \w{\Sigma:\cC \lra \cC} induces an
endofunctor
\begin{myeq}\label{eqief}
\Sigma:\NlnC \lra \NlnC
\end{myeq}
\noindent of $\Sigma$-algebras of left cubical balls satisfying
\begin{myeq}\label{eqiefp}
\Sigma \OO\sb{B}(b\sb{1}, \dotsc, b\sb{k}) = \OO\sb{B}(\Sigma b\sb{1}, \dotsc, 
\Sigma b\sb{k})
\end{myeq}
\noindent for each left cubical ball of dimension $\leq n$, see~\ref{dalgncb}(3).

\begin{defn}\label{dcsa}
Let \w{\Tln} be a $\Sigma$-track algebra and let
$$
\xymatrix@1{\Sigma:\cT(\leq n)\ar[r] & \cT(\leq n)}
$$
\noindent be an endofunctor of \w[,]{\cT(\leq n)} as in  \wref[,]{eqief} 
satisfying \wref[,]{eqiefp} such that $\Sigma$ induces the endofunctor
\w{\Sigma:\cA \lra \cA} of the $\Sigma$-algebra $\cA$. Then \w{\Tln}
is a \emph{complete $\Sigma$-algebra of left cubical balls} if the sum rule
and the product rule below are satisfied.
\end{defn}

\begin{defn}[sum rule]\label{defi:sumRule}
Let\w[.]{m \leq n} Assume given a pre-chain complex $L$ in \w{\cT(m)}
based on
\begin{equation*}
\tag{1} \xymatrix@1{Y & A\sb{0}\ar[l]_(.44)\alpha & 
A\sb{1}\ar[l]_(.4){\delta\sb{1}} &
\dotsc\ar[l] & A\sb{m + 1}\ar[l]}
\end{equation*}
and a pre-chain complex \w{L'} in \w{\cT(m - 1)} based on
\begin{equation*}
\tag{2} \xymatrix@1{Y & \Sigma A\sb{1}\ar[l]_\beta &
\Sigma A\sb{2}\ar[l]_(.42){\Sigma \delta\sb{2}} & \dotsc\ar[l] & 
\Sigma A\sb{m + 1}\ar[l]}
\end{equation*}
such that \w{L'} restricted to \w{\bZ(m + 1, 1)} coincides with \w[.]{\Sigma L}
Under these assumptions, the \emph{sum rule} in dimension $m$ demands that
there exist a pre-chain complex
\w{L''} in \w{\cT(m)} based on~(1) such that \w{L''} restricted to
\w{\bZ(m + 1, 0)} coincides with $L$ and
\begin{equation*}
\tag{3} \OO L''(m + 1, \partial I\sb{m + 1})=\OO L(m + 1, \partial I\sb{m + 1}) +
\OO L'(m + 1, \partial I\sb{m}).
\end{equation*}
\end{defn}

\begin{prop}\label{psrin}
The sum rule is satisfied in \w{\NlnC} in \wref[.]{eqief}
\end{prop}

\begin{proof}
Let \w{I \approx [0, 2] = I \cup I} be the homeomorphism of intervals carrying
\w{t \in I} to \w[.]{2t} Then we have:
\begin{align*}
\tag{1} I\sp{k + 1} = I \times I\sp{k} \approx (I \cup I) \times I\sp{k} =
I\sp{k + 1} \cup I\sp{k + 1}~.
\end{align*}
\noindent for each \w[.]{k \geq 0}

For each \w[,]{j = 1, \dotsc, m + 1} \w{L'} yields the left \wwb{j-1}cube
\w{a\sb{j} = L'(j + 1, I\sb{j - 1})} in \w[,]{\Mor\sb{\cC}(\Sigma A\sb{j + 1}, Y)}
which by \w{\tau\sb{\Sigma}} in Definition \ref{dsmac} yields the $j$-cube
\w{\overline a\sb{j}} in \w{\Mor\sb{\cC}(A\sb{j + 1}, Y)} adjoint to
\w[.]{\tau\sb{\Sigma} a\sb{j}}  Using~(1), we define the $j$ cube
$$
L''(j + 1, I\sb{j}) = L(j + 1, I\sb{j}) \cup \overline a\sb{j}~.
$$
\noindent This defines \w{L''} completely, since \w{L''} restricted to
\w{\bZ(m + 1, 0)} coincides with $L$. One can now check that the sum
formula~(\ref{defi:sumRule})~(3) holds.
\end{proof}

Let $\cT$ be an \wwb{n + k}category enriched in left cubical $(n + k)$-sets
and let \w{\cT\sp{n}} be the $n$-skeleton of $\cT$. Then \w{\cT\sp{n}}
is an $n$-graded category enriched in $n$-cubical sets. We consider a
pre-chain complex
$$
\xymatrix@1{R:\bZ(\infty, 0)_\otimes\sp{n}\ar[r] & \cT\sp{n}~.}
$$

\begin{defn}
A \emph{chain module with operators in $R$} is a functor $L$ which carries
a morphism \w{V:i \lra -1} \wb{i \geq 0} in \w{\bZ(\infty, -1)_\otimes\sp{n}}
to an element
$$
L(V) \in \Mor_\cT(R, Y)\sb{\dim(V) + k}
$$
\noindent such that the inclusion property
$$
L(V) = (d\sb{V, W} \otimes I\sp{k})\sp{\ast} L(W)
$$
\noindent holds if $V$ is in the boundary of $W$ and such that for a composite
\w{V \otimes V'} of morphisms in \w{\bZ(\infty, -1)_\otimes\sp{n}} the equation
$$
L(V \otimes V') = L(V) R(V')
$$
\noindent holds, where the right hand side denotes the composite in the
\wwb{n+k}category $\cT$.
\end{defn}

\begin{lemma}
A chain module $L$ with operators in $R$ is determined by the elements
$$
L(m, I\sb{m}) \in \Mor_\cT(R\sb{m}, Y)\sb{m + k}
$$
\noindent where \w[.]{I\sb{0} = \emptyset$ and $m = 0, \dotsc, n}
\end{lemma}

Now let \w{B = B\sb{1} \cup \dotsc \cup B\sb{s}} be a left cubical ball of 
dimension $k$ with cells \w{B\sb{i}} and gluing maps \w{d\sb{e, i}} as in 
Definition \ref{dlcb}. An $s$-tuple \w{(L\sb{1}, \dotsc, L\sb{s})} of 
chain modules \w{L\sb{i}} with operators in $R$ satisfies the 
\emph{gluing condition} in $B$ if for \w{V:m \to -1} we have
\begin{myeq}\label{eqcmgc}
(I\sp{\dim(V)} \times d\sb{e, i})\sp{\ast} L\sb{i}(V) = (I\sp{\dim(V)} \times
d\sb{e, j})\sp{\ast} L\sb{j}(V)~.
\end{myeq}

The left cubical ball \w{C = T\sb{0}\sp{n}} has cells
\w[.]{C\sb{1},\dotsc, C\sb{n + 1}} The product \w{B \times C} is a left cubical 
ball with cells \w[.]{B\sb{i} \times C\sb{j}} Let
\w[,]{(c\sb{1}, \dotsc, c\sb{n + 2}) = \partial I\sb{n + 1}} see \wref[.]{eqtom}

\begin{lemma}
Given $(L\sb{1}, \dotsc L\sb{s})$ as in \wref{eqcmgc} we obtain the tuple of
\wwb{m + k}cubes \wb{r = n + 1} \w{L\sb{i}(r, c\sb{j})} satisfying the gluing
condition in \w[,]{B \times C} so that the obstruction
$$
\OO\sb{B \times C}(L\sb{i}(r, c\sb{j}),\hsm i = 1, \dotsc, s\hsm \text{and}\hsm
j = 1, \dotsc, n + 1)\hsm \text{in}\hsm \Hom\sb{\cA}(\Sigma\sp{n + k} R\sb{r}, Y)
$$
\noindent is defined in the $\Sigma$-track algebra \w[.]{\cT = \cT(n + k)}
Also the tuple \w{L\sb{i}(0, \emptyset)} satisfies the gluing condition, so that
the obstruction
$$
\OO\sb{B}(L\sb{1}(0, \emptyset), \dotsc, L\sb{s}(0, \emptyset)) \in
\Hom\sb{\cA}(\Sigma\sp{k} R\sb{0}, Y)
$$
\noindent is defined in \w[.]{\cT(k)}
\end{lemma}

\begin{defn}[product rule]
Let $R$ and \w{L\sb{1}, \dotsc, L\sb{s}} be given as above where $R$ is based on
$$
\xymatrix@1{
\dotsc\ar[r] & R\sb{n + 1}\ar[r] & R\sb{n}\ar[r] & 
\dotsc\ar[r]\sp{\delta\sb{1}} & R\sb{0}~.
}
$$
\noindent Let \w{\alpha \in \Hom\sb{\cA}(\Sigma\sp{k}R\sb{0}, Y)} be given by
\w[.]{\OO\sb{B}(L\sb{i}(0, \emptyset))} Then there exists a pre-chain complex 
\w{L'} based on
$$
\xymatrix@1{\Sigma\sp{k} R\sb{n + 1}\ar[r] & \Sigma\sp{k} R\sb{n}\ar[r] &
\dotsc\ar[r]^(.4){\Sigma\sp{k} \delta\sb{1}} &
\Sigma\sp{k} R\sb{0}\ar[r]^(.56)\alpha & Y~,}
$$
\noindent such that the equation
$$
\OO\sb{C}(L'(n + 1, \partial I\sb{n + 1})) = 
\OO\sb{B \times C}(L\sb{i}(0, c\sb{j}))
$$
\noindent holds in \w[,]{\Hom\sb{\cA}(\Sigma\sp{k + n} R\sb{n + 1}, Y)} and \w{L'}
restricted to \w{\bZ(n + 1, 0)} coincides with \w[.]{\Sigma\sp{k} R} This is
the \emph{product rule} in \w[.]{\cT(n + k)}
\end{defn}

\begin{prop}\label{ppr}
Let $\cC$ be a complete $\Sigma$-mapping algebra. Then
\w{\Nul{\leq (n + k)}(\cC)} satisfies the product rule.
\end{prop}

\begin{proof}
We have the homeomorphism \w[,]{S\sp{k} = B/\partial B} so we can
replace \w{\Sigma\sp{k} R\sb{i}} by \w[.]{(B/\partial B)\wedge R\sb{i}} 
Gluing the various \w{L\sb{i}} yields \w[.]{L'}
\end{proof}

\begin{remark}
\label{rsigmaTrackAlgebraHigherExtGroups}
If  \w{\cT(\leq 2n)} is a complete $\Sigma$-algebra of left cubical balls
(Definition \ref{dcsa}) the higher \ww{\Ext}-groups \w{E\sb{m}\sp{r, s}} are
well-defined by Definition\ref{dobstr} for \w[.]{m = 2, \dotsc, n + 2}
A proof can be given along the lines of the argument given below to show
that \w[.]{d\sb{2}d\sb{2}=0}

Since a complete $\Sigma$-mapping algebra $\cC$ yields a complete
$\Sigma$-track algebra \w[,]{\Nul{\leq 2n}(\cC)} the higher \ww{\Ext}-groups
are well defined in this case.
\end{remark}

\begin{example}
As an application of the product rule, we prove that \w[:]{d\sb{2} d\sb{2} = 0}
by \wref{eqdaeg} we have a diagram
$$
\xymatrix{
A\sb{r + 4}\ar[r]\sp{\delta_4}\ar@{=}[d] &
A\sb{r + 3}\ar[r]\sp{\delta_3}\ar@{=}[d] &
A\sb{r + 2}
\ar[r]\sp{\delta\sb{2}}\ar@{=}[d]\ar@/^{1.9em}/[rr]\sb{\Uparrow G}\sp{o} &
A\sb{r + 1}\ar[r]\sp{\delta\sb{1}}\ar@/_{1.9em}/[rr]\sp{\Downarrow F}\sb{o} &
A\sb{r}\ar[r]\sp{\beta} & Y\\
R\sb{2} & R\sb{1} & R\sb{0}
}
$$
\noindent For \w[,]{B = T\sb{0}\sp{1} = B\sb{1} \cup B\sb{2}} we choose
\w{L\sb{1}(0, \emptyset) = F \delta\sb{2}} and
\w[,]{L\sb{2}(0, \emptyset) = \beta G} where $F$ and $G$ are $1$-cubes with
\w{G = K(r + 2, I\sb{1})} given by the resolution $K$. Then
\begin{align*}
\tag{1} \alpha = d\sb{2} \beta = \OO\sb{B}(L\sb{1}(0, \emptyset),
L\sb{2}(0, \emptyset)).
\end{align*}
\noindent Now let \w{C= T\sb{0}\sp{1} = C\sb{1} \cup C\sb{2}} and
\w[.]{(c\sb{1}, c\sb{2}) = (I\sb{1} \otimes \emptyset, \emptyset \otimes I\sb{1})}
Then \w{L\sb{i}(2,  c\sb{j})} is defined as follows:
$$
\begin{array}{ccccc}
L\sb{1}(2, c\sb{1}) & = & L\sb{1}(2, I\sb{1} \otimes \emptyset) & = &
F K(r + 3, I\sb{1}) \delta_4\\
L\sb{1}(2, c\sb{2}) & = & L\sb{1}(2, \emptyset \otimes I\sb{1}) & = &
F \delta\sb{2} K(r + 4, I\sb{1})\\
L\sb{2}(2, c\sb{1}) & = & L\sb{2}(2, I\sb{1} \otimes \emptyset) & = &
\beta K(r + 3, I\sb{2}) \delta_4\\
L\sb{2}(2, c\sb{2}) & = & L\sb{2}(2, \emptyset \otimes I\sb{1}) & = &
\beta K(r + 2, I\sb{1}) K(r + 4, I\sb{1})
\end{array}
$$

Now the product rule shows that
\begin{align*}
\tag{2} d\sb{2} d\sb{2} \beta = d\sb{2} \alpha = 
\OO\sb{B\times C}(L\sb{i}(0, c\sb{j})) = 0
\end{align*}
and the rules in an algebra of left $2$-cubical balls show that this obstruction
is trivial. In fact, we have
\begin{align*}
\tag{3} &~\OO(I\sb{1} I\sb{1} \emptyset, \emptyset I\sb{2} \emptyset,
I\sb{1} \emptyset I\sb{1}, \emptyset I\sb{1} I\sb{1})\\
\tag{4} =&~ \OO(I\sb{1} I\sb{1} \emptyset, \emptyset \emptyset I\sb{2}, I\sb{1}
\emptyset I\sb{1})\\
\tag{5} =&~0
\end{align*}
Here~(3) is the obstruction~(2), (4) is a consequence of the complement
rule, and
$$
\OO(\emptyset I\sb{2}, I\sb{1} I\sb{1}, I\sb{2} \emptyset)~=~0~,
$$
\noindent which follows from the fact that $K$ is a resolution. Naturality
yields
$$
\OO(\emptyset\emptyset I\sb{2},\emptyset I\sb{1}I\sb{1},\emptyset
I\sb{2}\emptyset) = 0~.
$$
\noindent Finally, (5) follows from the triviality rule.
\end{example}

\end{document}